\documentclass[12pt,a4paper,reqno]{amsart}

\usepackage[margin=2.8cm,footskip=1cm]{geometry}

\usepackage[utf8]{inputenc}
\usepackage[T1]{fontenc}
\usepackage[english]{babel}
\usepackage{amsmath}
\usepackage{amsfonts}
\usepackage{amssymb}
\usepackage{amsthm}
\usepackage{mathrsfs}
\usepackage{ucs}
\usepackage{graphicx}
\usepackage{xcolor}
\usepackage{dsfont}
\usepackage{tikz}
\usepackage{wasysym}
\usepackage{physics}
\usepackage{mathtools}
\mathtoolsset{showonlyrefs}	
\usepackage{xifthen}
\usepackage[textwidth=2.5cm,textsize=tiny]{todonotes}
\usepackage[inline]{enumitem}
\setlist[itemize]{label={\(\boldsymbol\cdot\)}}
\usepackage{stmaryrd}
\usepackage{constants}
\usepackage[ruled]{algorithm2e}
\usepackage{tikz,pgfplots}
\usepackage[hidelinks]{hyperref}
\usepackage[font={small}]{caption}

\graphicspath{{Pictures/}}

\usepackage{constants}

\newconstantfamily{ctrlcst}{
	symbol=\kappa
}
\newconstantfamily{csts}{
	symbol=C
}
\renewconstantfamily{normal}{
	symbol=c
}

\makeatletter
\newcommand{\pushright}[1]{\ifmeasuring@#1\else\omit\hfill\(\displaystyle#1\)\fi\ignorespaces}
\newcommand{\pushleft}[1]{\ifmeasuring@#1\else\omit\(\displaystyle#1\)\hfill\fi\ignorespaces}
\makeatother

\renewcommand{\norm}[1]{\|#1\|}
\newcommand{\normI}[1]{\left\|#1\right\|_{\scriptscriptstyle 1}}
\newcommand{\normIV}[1]{\left\|#1\right\|_{\scriptscriptstyle 4}}
\newcommand{\normsup}[1]{\left\|#1\right\|_{\scriptscriptstyle\infty}}
\newcommand{\setof}[2]{\{#1\,:\,#2\}}

\newcommand{\Bsetof}[2]{\Bigl\{#1\,:\,#2\Bigr\}}
\newcommand{\given}{\,|\,}

\newcommand{\comp}{\mathrm{c}}
\newcommand{\iid}{i.i.d.\xspace}
\newcommand{\Ham}{\mathscr{H}}
\newcommand{\Ising}{\mathrm{Ising}}
\newcommand{\IsingPosFie}{\mathrm{IPF}}
\newcommand{\Potts}{\mathrm{Potts}}
\newcommand{\FK}{\mathrm{FK}}
\newcommand{\Bern}{\mathrm{Bern}}

\newcommand{\XY}{\mathrm{XY}}
\newcommand{\GFF}{\mathrm{GFF}}
\newcommand{\SAW}{\mathrm{SAW}}
\newcommand{\saw}{\mathsf{SAW}}
\newcommand{\KRW}{\mathrm{KRW}}
\newcommand{\walk}{\mathsf{W}}

\newcommand{\sat}{\mathrm{sat}}

\newcommand{\IF}[1]{\mathds{1}_{\{#1\}}}

\newcommand{\scrN}{\mathscr{N}}
\newcommand{\UnitBall}{\mathscr{U}}
\newcommand{\Wulff}{\mathscr{W}}
\newcommand{\surcharge}{\mathfrak{s}}

\newcommand{\cone}{\mathscr{Y}}
\newcommand{\scrC}{\mathscr{C}}

\newcommand{\betac}{\beta_{\mathrm{c}}}
\newcommand{\lambdac}{\lambda_{\mathrm{c}}}
\newcommand{\lambdaqlr}{\lambda_{\mathrm{\mathrm{exp}}}}

\theoremstyle{plain}
\newtheorem{theorem}{Theorem}[section]
\newtheorem{lemma}[theorem]{Lemma}

\newtheorem{corollary}[theorem]{Corollary}
\newtheorem{conjecture}[theorem]{Conjecture}

\newtheorem{remark}{Remark}[section]

\newtheorem{claim}{Claim}
\newtheorem{open}[theorem]{Open problem}

\theoremstyle{definition}
\newtheorem{obs}{Observation}

\newcommand{\calT}{\mathcal{T}}

\newcommand{\bbG}{\mathbb{G}}

\newcommand{\bbN}{\mathbb{N}}

\newcommand{\bbR}{\mathbb{R}}
\newcommand{\bbS}{\mathbb{S}}

\newcommand{\bbZ}{\mathbb{Z}}

\newcommand{\sfo}{{\mathsf o}}

\newcommand{\sfB}{\mathsf{B}}

\newcommand{\sfO}{\mathsf{O}}

\newcommand{\be}[1]{\begin{equation}\label{#1}}
	\newcommand{\ee}{\end{equation}}

\renewcommand{\epsilon}{\varepsilon}

\author{Yacine Aoun}
\address{Section de Mathématiques, Université de Genève, CH-1211 Genève, Switzerland}
\email{Yacine.Aoun@unige.ch}
\author{Dmitry Ioffe}
\address{Faculty of IE\&M, Technion, Haifa 32000, Israel}
\email{ieioffe@ie.technion.ac.il}
\author{S\'{e}bastien Ott}
\address{Dipartimento di Matematica e Fisica, Univ. Roma Tre, 00146 Roma, Italy}
\email{ott.sebast@gmail.com}
\author{Yvan Velenik}
\address{Section de Mathématiques, Université de Genève, CH-1211 Genève, Switzerland}
\email{yvan.velenik@unige.ch}

\date{\today}

\title[Non-analyticity of the correlation length]{Non-analyticity of the correlation length\\in systems with\\ exponentially decaying interactions}

\begin{document}

\maketitle

\begin{abstract}
	We consider a variety of lattice spin systems (including Ising, Potts and XY models) on \(\mathbb{Z}^d\) with long-range interactions of the form \(J_x = \psi(x) e^{-|x|}\), where \(\psi(x) = e^{\sfo(|x|)}\) and \(|\cdot|\) is an arbitrary norm.
	
	We characterize explicitly the prefactors \(\psi\) that give rise to a correlation length that is not analytic in the relevant external parameter(s) (inverse temperature \(\beta\), magnetic field \(h\), etc).
	Our results apply in any dimension.
	
	As an interesting particular case, we prove that, in one-dimensional systems, the correlation length is non-analytic whenever \(\psi\) is summable, in sharp contrast to the well-known analytic behavior of all standard thermodynamic quantities.

	We also point out that this non-analyticity, when present, also manifests itself in a qualitative change of behavior of the 2-point function.
	In particular, we relate the lack of analyticity of the correlation length to the failure of the \emph{mass gap condition} in the Ornstein--Zernike theory of correlations.
\end{abstract}


\section{Introduction and results}


%
\subsection{Introduction.}
%

The correlation length plays a fundamental role in our understanding of the properties of a statistical mechanical system.
It measures the typical distance over which the microscopic degrees of freedom are strongly correlated.
The usual way of defining it precisely is as the inverse of the rate of exponential decay of the 2-point function.
In systems in which the interactions have an infinite range, the correlation length can only be finite if these interactions decay at least exponentially fast with the distance.
Such a system is then said to have \emph{short-range} interactions.\footnote{While the terminology ``short-range'' \textit{vs.} ``long-range'' appears to be rather unprecise, different authors meaning quite different things by these terms, there is agreement on the fact that interactions decreasing exponentially fast with the distance are short-range.}

It is often expected that systems with short-range interactions all give rise to qualitatively similar behavior.
This then serves as a justification for considering mainly systems with nearest-neighbor interactions as a (hopefully generic) representant of this class.

As a specific example, let us briefly discuss one-dimensional systems with short-range interactions.
For those systems, the pressure as well as all correlation functions are always analytic functions of the interaction parameters.
A proof for interactions decaying at least exponentially fast was given by Ruelle~\cite{Ruelle-1975}, while the general case of interactions with a finite first moment was settled by Dobrushin~\cite{Dobrushin-1974} (see also~\cite{Cassandro+Olivieri-1981}).
This is known \emph{not} to be the case, at least for some systems, for interactions decaying
even slower with the distance~\cite{Dyson-1969, Frohlich+Spencer-1982}.

\medskip
In the present work, we consider a variety of lattice systems with exponentially decaying interactions.
We show that, in contrast to the expectation above, such systems can display qualitatively different behavior \emph{depending on the properties of the sub-exponential corrections}.

Under weak assumptions, the correlation length associated with systems whose interactions decay faster than any exponential tends to zero as the temperature tends to infinity.
In systems with exponentially decaying interactions, however, this cannot happen: indeed, the rate of exponential decay of the 2-point function can never be larger than the rate of decay of the interaction.
This suggests that, as the temperature becomes very large, one of the two following scenarii should occur: either there is a temperature \(T_{\sat}\) above which the correlation length becomes constant, or the correlation length asymptotically converges, as \(T\to\infty\), to the inverse of the rate of exponential decay of the interaction.
Notice that when the first alternative happens, the correlation length cannot be an analytic function of the temperature.

It turns out that both scenarios described above are possible.
In fact, both can be realized in the same system by considering the 2-point function in different directions.
What determines whether saturation (and thus non-analyticity) occurs is the correction to the exponential decay of the interactions.
We characterize explicitly the prefactors that give rise to saturation of the correlation length as a function of the relevant parameter (inverse temperature \(\beta\), magnetic field \(h\), etc).
Our analysis also applies to one-dimensional systems, thereby showing that the correlation length of one-dimensional systems with short-range interactions can exhibit a non-analytic behavior, in sharp contrast with the standard analyticity results mentioned above.

We also relate the change of behavior of the correlation length to a violation of the mass gap condition in the theory of correlations developed in the early 20th Century by Ornstein and Zernike, and explain how this affects the behavior of the prefactor to the exponential decay of the 2-point function.

%
\subsection{Convention and notation}
%

In this paper, \(|\cdot|\) denotes some arbitrary norm on \(\bbR^d\), while we reserve \(\|\cdot\|\) for the Euclidean norm.
The unit sphere in the Euclidean norm is denoted \(\bbS^{d-1}\). Given \(x\in\bbR^d\), \([x]\) denotes the (unique) point in \(\bbZ^d\) such that \(x\in [x]+[-\frac12,\frac12)^d\).
To lighten notation, when an element \(x\in\bbR^d\) is treated as an element of \(\bbZ^d\), it means that \([x]\) is considered instead.

%
\subsection{Framework and models}
%

For simplicity, we shall always work on \(\bbZ^d\), but the methods developed in this paper should extend in a straightforward manner to more general settings.
We consider the case where the interaction strength between two lattice sites \(i,j\) is given by \(J_{ij}=J_{i-j}=\psi(i-j)e^{-|i-j|}\), where \(|\cdot|\) is some norm on \(\bbR^d\); we shall always assume that both \(\psi\) and \(|\cdot|\) are invariant under lattice symmetries. 
We will suppose \(\psi(y) >0\) for all \(y\neq 0\) to avoid technical issues.
We moreover require that \(\psi\) is a sub-exponential correction, that is,
\begin{equation}
	\lim_{|y|\to\infty} \frac{1}{|y|} \log(\psi(y)) =0.
	\label{eq:psi_subexp}
\end{equation}

The approach developed in this work is rather general and will be illustrated on various lattice spin systems and percolation models.
We will focus on suitably defined \emph{2-point functions} \(G_\lambda(x,y)\) (sometimes truncated), where \(\lambda\) is some external parameter.
We define now the various models that will be considered and give, in each case, the corresponding definition of \(G_\lambda\) and of the parameter \(\lambda\).

The following notation will occur regularly:
\begin{gather*}
	\bar{J} = \sum_{x\in\bbZ^d} J_{0x}, \quad P(x) = J_{0x}/\bar{J}.
\end{gather*}
By convention, we set \(\bar{J} = 1\) (and thus \(P(x) = J_{0x}\)), since the normalization can usually be absorbed into the inverse temperature or in a global scaling of the field, and assume that \(J_{00} =0\) (so \(\bar{J} = \sum_{x\in\bbZ^d\setminus\{0\}} J_{0x} = 1\)).
All models will come with a parameter (generically denoted \(\lambda\)). They also all have a natural transition point \(\lambdac\) (possibly at infinity) where the model ceases to be defined or undergoes a drastic change of behavior.

We will always work in a regime \(\lambda\in[0, \lambdaqlr)\) where \(\lambdaqlr\leq \lambdac\) is the point at which (quasi-)long range order occurs (see~\eqref{eq:lambdaqlr_def}) for the model. For all models under consideration, it is conjectured that \(\lambdaqlr = \lambdac\).

\subsubsection{KRW model}

A walk is a finite sequence of vertices \((\gamma_0, \dots, \gamma_n)\) in \(\bbZ^d\). The length of \(\gamma\) is \(\abs{\gamma} =n\). Let \(\walk(x,y)\) be the set of (variable length) walks with \(\gamma_0=x,\gamma_{\abs{\gamma}} = y\).
The 2-point function of the killed random walk (KRW) is defined by
\begin{equation}
	G^{\KRW}_{\lambda}(x,y) = \sum_{\gamma\in\walk(x,y)} \prod_{i=1}^{\abs{\gamma}} \lambda J_{\gamma_{i-1} \gamma_i}.
\end{equation}
\(\lambdac\) is defined by
\begin{equation*}
	\lambdac= \sup\Bsetof{\lambda\geq 0}{\sum_{x\in\bbZ^d} G^{\KRW}_{\lambda}(0,x) <\infty}.
\end{equation*}
Our choice of normalization for \(J\) implies that \(\lambdac=1\).

\subsubsection{SAW model}

Self-Avoiding Walks are finite sequences of vertices \((\gamma_0, \dots, \gamma_n)\) in \(\bbZ^d\) with at most one instance of each vertex (that is, \(i\neq j\implies\gamma_i\neq\gamma_j\)).
Denote \(\abs{\gamma} = n\) the length of the walk.
Let \(\saw(x,y)\) be the set of (variable length) SAW with \(\gamma_0=x,\gamma_{\abs{\gamma}}=y\).
We then let
\begin{equation}
	G^{\SAW}_{\lambda}(x,y) = \sum_{\gamma\in\saw(x,y)} \prod_{i=1}^{\abs{\gamma}} \lambda J_{\gamma_{i-1} \gamma_i}.
\end{equation}
\(\lambdac\) is defined by
\begin{equation*}
	\lambdac= \sup\Bsetof{\lambda\geq 0}{\sum_{x\in\bbZ^d} G^{\SAW}_{\lambda}(0,x) <\infty}.
\end{equation*}
Since \(G^{\SAW}_{\lambda}(x,y) \leq G^{\KRW}_{\lambda}(x,y)\), it follows that \(\lambdac^{\SAW}\geq \lambdac^{\KRW} = 1\).

\subsubsection{Ising model}

The Ising model at inverse temperature \(\beta\geq 0\) and magnetic field \(h\in\bbR\) on \(\bbZ^d\) is the probability measure on \(\{-1,+1\}^{\bbZ^d}\) given by the weak limit of the finite-volume measures (for \(\sigma\in\{-1,+1\}^{\Lambda_N}\) and \(\Lambda_N=[-N,N]^{d}\cap\mathbb{Z}^{d}\)).
\[
	\mu^{\Ising}_{\Lambda_N;\beta,h}(\sigma) = \frac{1}{Z_{\Lambda_N;\beta,h}^{\Ising}} e^{-\beta\Ham_N(\sigma)},
\]
with Hamiltonian
\[
	\Ham_N(\sigma) = -\sum_{\{i,j\}\subset\Lambda_N } J_{ij} \sigma_i\sigma_j - h\sum_{i\in\Lambda_N}\sigma_i
\]
and partition function \(Z_{\Lambda_N;\beta,h}^{\Ising}\).
The limit \(\mu^{\Ising}_{\beta,h}=\lim_{N\to\infty}\mu^{\Ising}_{\Lambda_N;\beta,h}\) is always well defined and agrees with the unique infinite-volume measure whenever \(h\neq 0\) or \(\beta<\betac\), the critical point of the model.

For this model, we will consider two different situations, depending on which parameter we choose to vary:
\begin{itemize}
	\item When \(h=0\), we consider
	\begin{equation}
		G^{\Ising}_{\beta}(x,y) = \mu^{\Ising}_{\beta,0}(\sigma_x\sigma_y)
		\quad\text{ and }\quad
		\lambda = \beta.
	\end{equation}
	In this case, \(\lambdac = \betac(d)\) marks the boundary of the high-temperature regime (\(\lim_{\norm{x}\to\infty}\mu^{\Ising}_{\beta,0}(\sigma_0\sigma_x) =0\) for \(\beta< \betac\) and is \(>0\) for \(\beta>\betac\)).
	\item When \(h>0\), we allow arbitrary values of \(\beta\geq 0\) and consider
	\begin{equation}
		G^{\IsingPosFie}_{\beta,h}(x,y) = \mu^{\Ising}_{\beta,h}(\sigma_x\sigma_y) - \mu^{\Ising}_{\beta,h}(\sigma_x)\mu^{\Ising}_{\beta,h}(\sigma_y)
		\quad\text{ and }\quad
		\lambda = e^{-h}.
	\end{equation}
	Of course, here \(\lambdac=1\).
	The superscript \(\IsingPosFie\) stands for ``Ising with a Positive Field''.
\end{itemize}

\subsubsection{Lattice GFF}

The lattice Gaussian Free Field with mass \(m\geq 0\) on \(\bbZ^d\) is the probability measure on \(\bbR^{\bbZ^d}\) given by the weak limit of the finite-volume measures (for \(\sigma\in\bbR^{\Lambda_N}\))
\[
	\dd\mu^{\GFF}_{m,\Lambda_N}(\sigma) = \frac{1}{Z_{m,\Lambda_N}^{\GFF}} e^{-\Ham_N(\sigma)-m^2\sum_{i\in\Lambda_N}\sigma_i^2 } \,\dd\sigma,
\]
with Hamiltonian
\[
	\Ham_N(\sigma) = -\sum_{\{i,j\}\subset\Lambda_N } J_{ij} (\sigma_i-\sigma_j)^2
\]
and partition function \(Z_{m,\Lambda_N}^{\GFF}\). Above, \(\dd\sigma\) denotes the Lebesgue measure on \(\bbR^{\Lambda_N}\).
The limit \(\mu^{\GFF}_{m}=\lim_{N\to\infty}\mu^{\GFF}_{m,\Lambda_N}\) exists and is unique for any \(m>0\).
When considering the measure at \(m=0\), we mean the measure \(\mu^{\GFF}=\lim_{m\downarrow 0} \mu^{\GFF}_{m}\).
The latter limit exists when \(d\geq 3\), but not in dimensions \(1\) and \(2\).

For this model, we define
\begin{equation}
	G^{\GFF}_{(1+m^2)^{-1}}(x,y) = \mu^{\GFF}_{m}(\sigma_x\sigma_y),\quad \lambda = \frac{1}{1+ m^2}.
\end{equation}
The 2-point function of the GFF has a nice probabilistic interpretation: let \(P\) be the probability measure on \(\bbZ^d\) given by \(P(x)=J_{0x}\).
Let \(P_x^{m}=P_{J,x}^m\) denote the law of the random walk started at \(x\) with killing \(\frac{m^2}{1+m^2}\) and \textit{a priori} \iid steps of law \(P\) and let \(E_x^m\) be the corresponding expectation.
Let \(X_i\) be the \(i\)th step and \(S_0=x,\ S_k= S_{k-1}+ X_k\) be the position of the walk at time \(k\).
Denote by \(T\) the time of death of the walk. One has \(P^m(T=k)=(1+m^2)^{-k} m^2\).
The 2-point function can then be expressed as
\begin{equation}\label{eq:GFF_RW_Rep_Cov}
	G_{\lambda}^{\GFF}(x,z) = \frac{1}{1+m^2}E_x^m\Big[\sum_{k=0}^{T-1} \IF{S_k=z}\Big].
\end{equation}
Thanks to the normalization \(\bar{J}=1\), it is thus directly related to the \(\KRW\) via the identity
\begin{equation}\label{eq:GFF_to_KRW}
	G_{\lambda}^{\GFF}(x,z) = \lambda G_{\lambda}^{\KRW}(x,z).
\end{equation}
In particular, \(\lambdac = 1\) (which corresponds to \(m=0\)) and
\(\sup_{x\in\bbZ^d} G_{\lambda}^{\GFF}(0,x) < \infty\) for all \(\lambda \in [0,\lambdac)\) in any dimension.

\subsubsection{Potts model and FK percolation}

The \(q\)-state Potts model at inverse temperature \(\beta\geq 0\) on \(\bbZ^d\) with free boundary condition is the probability measure on \(\{1, 2, \dots, q\}^{\bbZ^d}\) (\(q\geq 2\)) given by the weak limit of the finite-volume measures (for \(\sigma\in\{1, \dots, q\}^{\Lambda_N}\))
\[
	\mu^{\Potts}_{\Lambda_N;\beta,q}(\sigma) = \frac{1}{Z_{\Lambda_N;\beta,q}^{\Potts}} e^{-\beta\Ham_N(\sigma)}
\]
with Hamiltonian
\[
	\Ham_N(\sigma) = -\sum_{\{i,j\}\subset\Lambda_N } J_{ij} \IF{\sigma_i=\sigma_j}
\]
and partition function \(Z_{\Lambda_N;\beta,q}^{\Potts}\).
We write \(\mu^{\Potts}_{\beta,q} =\lim_{N\to\infty} \mu^{\Potts}_{\Lambda_N;\beta,q}\); this limit can be shown to exist.
From now on, we omit \(q\) from the notation, as in our study \(q\) remains fixed, while \(\beta\) varies.

For this model, we consider
\begin{equation}
	G^{\Potts}_{\beta}(x,y) = \mu^{\Potts}_{\beta}(\IF{\sigma_x=\sigma_y})- 1/q
	\quad\text{ and }\quad
	\lambda = \beta.
\end{equation}
As in the Ising model, we are interested in the regime \(\beta < \betac\), where \(\betac\) is the inverse temperature above which long-range order occurs (that is, \(\inf_{x}G^{\Potts}_{\beta}(0,x)>0\) for all \(\beta>\betac\), see below). We thus again have \(\lambdac=\betac(q,d)\).

One easily checks that the Ising model (with \(h=0\)) at inverse temperature \(2\beta\) corresponds to the \(2\)-state Potts model at inverse temperature \(\beta\).

Intimately related to the Potts model is the FK percolation model.
The latter is a measure on edge sub-graphs of \((\bbZ^d, E_d)\), where \(E_d=\bigl\{\{i,j\}\subset\bbZ^d\bigr\}\), depending on two parameters \(\beta\in\bbR_{\geq 0}\) and \(q\in\bbR_{>0}\), obtained as the weak limit of the finite-volume measures
\begin{equation}
	\Phi^{\FK}_{\Lambda_N;\beta,q}(\omega) = \frac{1}{Z^{\FK}_{\Lambda_N;\beta,q}} \prod_{\{i,j\}\in\omega}(e^{\beta J_{ij}}-1) q^{\kappa(\omega)},
\end{equation}
where \(\kappa(\omega)\) is the number of connected components in the graph with vertex set \(\Lambda_N\) and edge set \(\omega\) and \(Z^{\FK}_{\Lambda_N;\beta,q}\) is the partition function.
In this paper, we always assume that \(q\geq 1\). We use the superscript \(\Bern\) for the case \(q=1\) (Bernoulli percolation).
When \(q\in\bbN\) with \(q\geq 2\), one has the correspondence
\begin{equation}
	\label{eq:Potts_FK_Corresp}
	\mu^{\Potts}_{\beta,q}(\IF{\sigma_x=\sigma_y})- \frac{1}{q} = \frac{q-1}{q} \, \Phi^{\FK}_{\beta,q}(x\leftrightarrow y).
\end{equation}
For the FK percolation model, we consider
\begin{equation}
	G^{\FK}_{\beta}(x,y) = \Phi^{\FK}_{\beta,q}(x\leftrightarrow y)
	\quad\text{ and }\quad
	\lambda = \beta,
\end{equation}
where \(\{x\leftrightarrow y\}\) is the event that \(x\) and \(y\) belong to the same connected component.
As for the Potts model, \(\lambdac=\betac(q,d)\); here, this corresponds to the value at which the percolation transition occurs.

\subsubsection{XY model}

The XY model at inverse temperature \(\beta\geq 0\) on \(\bbZ^d\) is the probability measure on \((\bbS^{1})^{\bbZ^d}\) given by the weak limit of the finite-volume measures (for \(\theta\in[0,2\pi)^{\Lambda_N}\))
\[
	\dd\mu^{\XY}_{\Lambda_N;\beta}(\theta) = \frac{1}{Z_{\Lambda_N;\beta}^{\XY}} e^{-\beta\Ham_N(\theta)}\,\dd\theta
\]
with Hamiltonian
\[
	\Ham_N(\theta) = -\sum_{\{i,j\}\subset\Lambda_N } J_{ij}\cos(\theta_i-\theta_j)
\]
and partition function \(Z_{\Lambda_N;\beta}^{\XY}\).

In this case, we consider
\begin{equation}
	G_{\beta}^{\XY}(x,y) = \mu^{\XY}_{\beta}\big(\cos(\theta_x-\theta_y)\big)
	\quad\text{ and }\quad
	\lambda = \beta.
\end{equation}
In dimension \(1\) and \(2\), \(\lambdac\) is the point at which quasi-long-range order occurs (failure of exponential decay; in particular, \(\lambdac =\infty\) when \(d=1\)).
In dimension \(d\geq 3\), we set \(\lambdac=\betac^{\XY}(d)\) the inverse temperature above which long-range order occurs (spontaneous symmetry breaking).

%
\subsection{Inverse correlation length}
%

To each model introduced in the previous subsection, we have associated a suitable 2-point function \(G_\lambda\) depending on a parameter \(\lambda\) (for instance, \(\lambda=(1+m^2)^{-1}\) for the GFF and \(\lambda=\beta\) for the Potts model).
Each of these 2-point functions gives rise to an \emph{inverse correlation length} associated to a direction \(s\in\bbS^{d-1}\) via
\begin{equation*}
	\nu_s(\lambda) = -\lim_{n\to\infty} \frac{1}{n} \log G_{\lambda}(0,ns) .
\end{equation*}
This limit can be shown to exist in all the models considered above in the regime \(\lambda\in[0,\lambdac)\).
When highlighting the model under consideration, we shall write, for example, \(\nu_s^{\Ising}(\lambda)\).

We also define \(\lambdaqlr\) as
\begin{equation}
\label{eq:lambdaqlr_def}
\lambdaqlr = \min\bigl(\lambdac,\inf\setof{\lambda\geq 0}{\inf_s \nu_s(\lambda)=0}\bigr).
\end{equation}
(Let us note that the infimum over \(s\) is actually not required in this definition, as follows from Lemma~\ref{lem:rate_equiv_directions} below.) It marks the boundary of the regime in which \(\nu\) is non-trivial. It is often convenient to extend the function \(s\mapsto\nu_s(\lambda)\) to a function on \(\bbR^d\) by positive homogeneity. In all the models we consider, the resulting function is convex and defines a norm on \(\bbR^d\) whenever \(\lambda<\lambda_{\exp}\).
These and further basic properties of the inverse correlation length are discussed in Section~\ref{sec:BasicPropICL}.

\smallskip
The dependence of \(\nu_s(\lambda)\) in the parameter \(\lambda\) is the central topic of this paper.

%
\subsection{Mass gap, a comment on the Ornstein--Zernike theory}\label{sec:RemOZ}
%

For off-critical models, the Ornstein--Zernike (OZ) equation is an identity satisfied by \(G_{\lambda}\), first postulated by Ornstein and Zernike (initially, for high-temperature gases):
\begin{equation}
	\label{eq:OZ}
	G_{\lambda}(0,x) = D_{\lambda}(0,x) + \sum_{y} G_{\lambda}(y,x)D_{\lambda}(0,y) ,
\end{equation}
where \(D_{\lambda}\) is the direct correlation function (this equation can be seen as \emph{defining} \(D_{\lambda}\)), which is supposed to behave like the interaction: \(D_{\lambda}(x,y) \simeq J_{xy}\).
On the basis of~\eqref{eq:OZ}, Ornstein and Zernike were able to predict the sharp asymptotic behavior of \(G_{\lambda}\), provided that the following \emph{mass gap hypothesis} holds:
there exists \(c=c(\lambda)>0\) such that
\begin{equation*}
	D_{\lambda}(0,x) \leq e^{-c|x|} G_{\lambda}(0,x).
\end{equation*}
This hypothesis is supposed to hold in a vast class of high-temperature systems with finite correlation length.
One of the goals of the present work is to show that this hypothesis is doomed to \emph{fail} in certain simple models of this type at very high temperature and to provide some necessary conditions for the presence of the mass gap.

To be more explicit, in all models considered, we have an inequality of the form \(G_{\lambda}(0,x)\geq C J_{0x} = C\psi(x)e^{-|x|}\).
In particular, this implies that \(\nu_s \leq |s|\) for all \(s\in\bbS^{d-1}\).
We will study conditions on \(\psi\) and \(\lambda\) under which the inequality is either strict (``mass gap'') or an equality (saturation).
We will also be concerned with the asymptotic behavior of \(G_{\lambda}\) in the latter case, while the ``mass gap'' pendant of the question will only be discussed for the simplest case of \(\KRW\), the treatment of more general systems being postponed to a forthcoming paper.

A useful consequence of the OZ-equation~\eqref{eq:OZ}, which is at the heart of the derivation of the OZ prefactor, is the following (formal) identity
\begin{equation*}
	\label{eq:OZ_paths}
	G_{\lambda}(0,x) = \sum_{\gamma\in\walk(0,x)}\prod_{i=1}^{|\gamma|} D_{\lambda}(\gamma_{i-1},\gamma_i).
\end{equation*}One can see Simon-Lieb type inequalities
\begin{equation*}
	G_{\lambda}(0,x) \leq D_{\lambda}(0,x) + \sum_{y} G_{\lambda}(y,x)D_{\lambda}(0,y),
\end{equation*}as approaching the OZ equation. In particular, this inequality with \(D_{\lambda}(0,y)\simeq J_{xy}\) is directly related to our assumption~\ref{hyp:weak_SL} below.

%
\subsection{A link with condensation phenomena}
%

Recall that the (probabilistic version of) condensation phenomena can be summarized as follows:
take a family of real random variables \(X_1, \ldots, X_N\) (with \(N\) possibly random) and constrain their sum to take a value much larger than \(E[\sum_{k=1}^N X_i]\).
Condensation occurs if most of the deviation is realized by a single one of the \(X_k\)s.
In the case of condensation, large deviation properties of the sum are ``equivalent'' to those of the maximum (see, for instance, \cite{Godreche-2019} and references therein for additional information).

In our case, one can see the failure of the mass gap condition as a condensation transition: suppose the OZ equation holds.
\(G(0,x)\) is then represented as a sum over paths of some path weights.
The exponential cost of a path going from \(0\) to \(x\) is always at least of the order \(|x|\).
Once restricted to paths with exponential contribution of this order, the geometry of typical paths will be governed by a competition between entropy (combinatorics) and the sub-exponential part \(\psi\) of the steps weight.
In the mass gap regime, typical paths are constituted of a number of microscopic steps growing linearly with \(\|x\|\): in this situation, entropy wins over energy and the global exponential cost per unit length is decreased from \(|s|\) to some \(\nu_s<|s|\).
One recovers then the behavior of \(G\) predicted by Ornstein and Zernike.
In contrast, in the saturated regime, typical paths will have one giant step (a condensation phenomenon) and the behavior of \(G\) is governed by this kind of paths, which leads to \(G(0,x)\simeq D(0,x)\simeq J_{0x}\).

%
\subsection{Assumptions}
%

To avoid repeating the same argument multiple times, we shall make some assumptions on \(G_\lambda\) and prove the desired results based on those assumptions only (basically, we will prove the relevant claims for either \(\KRW\) or \(\SAW\) and the assumptions allow a comparison with those models).
Proofs (or reference to proofs) that the required properties hold for the different models we consider are collected in Appendix~\ref{app:Properties}.

\medskip
\begin{enumerate}[label={\ensuremath{\mathrm{[A_\arabic*]}}}, start=0]
	\item \label{hyp:G_Bounded_Pos_nu_pos}
		For any \(\lambda\in [0, \lambdac)\), \(G_{\lambda}(x,y)\geq 0\) for any \(x, y\in \bbZ^d\) \text{ and }\({\sup_{x\in\bbZ^d} G_{\lambda}(0,x) < \infty}\).
	\item \label{hyp:sub_mult}
		For any \(\lambda \in [0, \lambdac)\), there exists \(a_{\lambda} > 0\) such that, for any \(x, y, z\in\bbZ^d\),
	\begin{equation}
		\label{eq:G_sub_mult}
		G_{\lambda}(x,y) \geq a_{\lambda} G_{\lambda}(x,z) G_{\lambda}(z,y).
	\end{equation}This property holds at \(\lambdac\) if \(\sup_{x}G_{\lambdac}(0,x)<\infty\).
	\item \label{hyp:left_cont}
		For any \(x, y\in \bbZ^d\), \(\lambda\mapsto G_{\lambda}(x,y)\) is non-decreasing and left-continuous on \([0, \lambdac)\). This continuity extends to \([0,\lambdac]\) if \(G_{\lambdac}(x,y)\) is well defined.
	\item \label{hyp:weak_SL}
		There exists \(\alpha\geq 0\) such that, for any \(0\leq \lambda < \lambdac\), there exists \(C\geq 0\) such that for any \(x, y\in\bbZ^d\),
		\begin{equation}
			\label{eq:weak_SL}
			G_{\lambda}(x,y) \leq C G_{\alpha\lambda}^{\KRW}(x,y).
		\end{equation}
	\item \label{hyp:J_path_lower_bnd}
		For any \(\lambda \in [0, \lambdac)\), there exist \(c_\lambda>0\) and \(C_{\lambda}>0\) such that, for any collection \(\Gamma\subset\saw(x,y)\), one has
		\begin{equation}\label{eq:J_path_lower_bnd}
			G_{\lambda}(x,y) \geq c_\lambda \sum_{\gamma\in\Gamma}(C_{\lambda})^{\abs{\gamma}}\prod_{k=1}^{\abs{\gamma}} 	J_{\gamma_{k-1} \gamma_k}.
		\end{equation}
\end{enumerate}

\medskip
Our choice of \(\lambdac\)  and of \(G_{\lambda}\) ensures that~\ref{hyp:G_Bounded_Pos_nu_pos} is always satisfied.
Assumption~\ref{hyp:sub_mult} holds as soon as the model enjoys some GKS or FKG type inequalities. Assumption~\ref{hyp:left_cont} is often a consequence of the monotonicity of the Gibbs state with respect to \(\lambda\).
The existence of a well-defined high-temperature regime (or rather the \emph{proof} of its existence) depends on this monotonicity.
Assumption~\ref{hyp:weak_SL} is directly related to the Ornstein--Zernike equation~\eqref{eq:OZ} in the form given in~\eqref{eq:OZ_paths}. It is easily deduced from a weak form of Simon--Lieb type inequality, see Section~\ref{sec:RemOZ}.
Assumption~\ref{hyp:J_path_lower_bnd} may seem to be a strong requirement but is usually a consequence of a path representation of correlation functions, some form of which is available for vast classes of systems.

\bigskip
Part of our results will also require the following additional regularity assumption on the prefactor \(\psi\):
\begin{enumerate}[label={\ensuremath{\mathrm{[H_\arabic*]}}}, start=0]
	\item \label{hyp:PsiQuasiIsotropic}
	There exist \(C_\psi^+, C_\psi^- > 0\) and \(\psi_0:\bbN_{>0}\to\bbR\) such that, for all \(y\in\bbZ^d\setminus\{0\}\),
	\[
		C^-_\psi \psi_0(\normI{y}) \leq \psi(y) \leq C^+_\psi \psi_0(\normI{y}).
	\]
\end{enumerate}

%
\subsection{Surcharge function}
%

Our study has two ``parameters'': the prefactor \(\psi\), and the norm \(|\cdot|\).
It will be convenient to introduce a few quantities associated to the latter.

First, two convex sets are important: the unit ball \(\UnitBall\subset \bbR^d\) associated to the norm \(|\cdot|\) and the corresponding \emph{Wulff shape}
\[
	\Wulff = \setof{t\in\bbR^d}{\forall x\in\bbR^d,\, t\cdot x \leq |x|}.
\]
Given a direction \(s\in \bbS^{d-1}\), we say that the vector \(t\in\bbR^d\) is dual to \(s\) if
\(t\in\partial\Wulff\) and \(t\cdot s = |s|\). A direction \(s\) possesses a unique dual vector \(t\) if and only if \(\Wulff\) does not possess a facet with normal \(s\). Equivalently, there is a unique dual vector when the unit ball \(\UnitBall\) has a unique supporting hyperplane at \(s/|s|\). (See Fig.~\ref{fig:duality} for an illustration.)

\begin{figure}[ht]
	\includegraphics{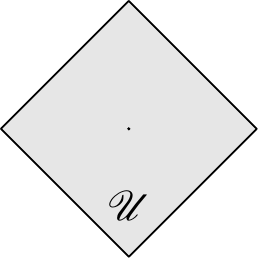}
	\hspace*{1cm}
	\includegraphics{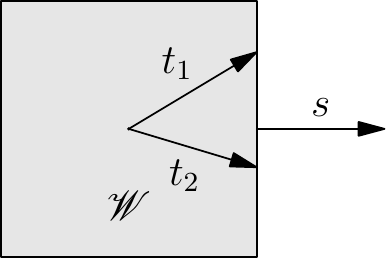}
	\hspace*{1cm}
	\includegraphics{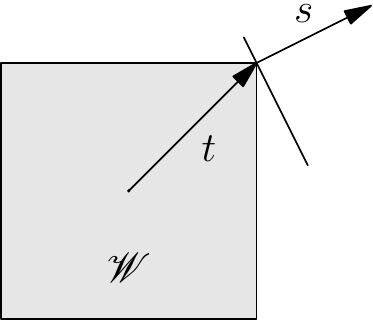}
	\caption{Left: The unit ball for the norm \(|\cdot|=\normI{\cdot}\). Middle: the corresponding Wulff shape \(\Wulff\) with two vectors \(t_1\) and \(t_2\) dual to \(s=(1,0)\). Right: the set \(\Wulff\) with the unique vector \(t\) dual to \(s=\frac{1}{\sqrt{5}}(2,1)\).}
	\label{fig:duality}
\end{figure}

The \emph{surcharge function} associated to a dual vector \(t\in\partial\Wulff\) is then defined by
\begin{equation*}
	\surcharge_t(x) = |x|- x\cdot t.
\end{equation*}
It immediately follows from the definition that \(\surcharge_t(x)\geq 0\) for all \(x\in\bbZ^d\) and \(\surcharge_t(s)=0\) if \(t\) is a vector dual to \(s\).

The surcharge function plays a major role in the Ornstein--Zernike theory as developed in~\cite{Campanino+Ioffe-2002,Campanino+Ioffe+Velenik-2003,Campanino+Ioffe+Velenik-2008}.
Informally, \(\surcharge_t(s')\) measures the additional cost (per unit length) that a step in direction \(s'\) incurs when your goal is to move in direction \(s\).
As far as we know, it first appeared, albeit in a somewhat different form, in~\cite{Alexander-1990}.

%
\subsection{Quasi-isotropy}\label{sec:QuasiIsotropy}
%

Some of our results hinge on a further regularity property of the norm \(|\cdot|\).

Let \(s\in\bbS^{d-1}\) and \(t\) be a dual vector.
Write \(s_0 = s/|s|\in\partial\UnitBall\) and \(\hat{t} = t/\norm{t} \in \bbS^{d-1}\). Let \(T_{s_0}\UnitBall\) be the tangent hyperplane to \(\UnitBall\) at \(s_0\) with normal \(\hat{t}\) (seen, as usual, as a vector space). It is always possible to choose the dual vector \(t\) such that the following holds (we shall call such a \(t\) \emph{admissible}\footnote{When there are multiple tangent hyperplanes to \(\partial\UnitBall\) at \(s_0\), convexity and symmetry imply that all non-extremal elements of the normal cone are admissible.}).
There exist \(\epsilon > 0\) and a neighborhood \(\scrN\) of \(s_0\) such that \(\partial\UnitBall\cap\scrN\) can be parametrized as (see Fig.~\ref{fig:paramU})
\[
\partial\UnitBall \cap \scrN = \setof{s_0 + \tau v - f(\tau v)\hat{t}}{v\in T_{s_0}\UnitBall\cap\bbS^{d-1},\, |\tau|<\epsilon},
\]
for some convex nonnegative function \(f:T_{s_0} \to \bbR\) satisfying \(f(0)=0\).
\begin{figure}
	\centering
	\includegraphics{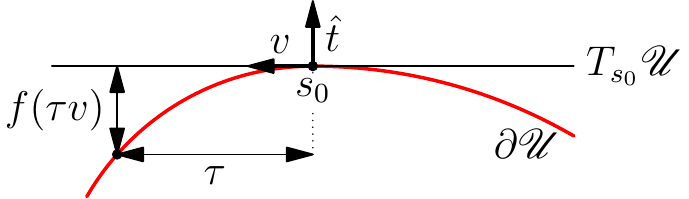}
	\caption{The local parametrization of \(\partial\UnitBall\) in a neighborhood of \(s_0\).}
	\label{fig:paramU}
\end{figure}

\medskip
We will say that \(\partial\UnitBall\) is \emph{quasi-isotropic} in direction \(s\) if the qualitative behavior of \(f\) is the same in all directions \(v\): there exist \(c_+\geq c_- > 0\) and an non-decreasing non-negative convex function \(g\) such that, for all \(v\in T_{s_0}\UnitBall\cap\bbS^{d-1}\) and all \(\tau\in (0, \epsilon)\),
\begin{equation}\label{eq:QuasiIsotropy}
	c_+ g(\tau)\geq f(\tau v) \geq c_- g(\tau) .
\end{equation}
Taking \(\scrN\) and \(\epsilon\) smaller if necessary, we can further assume that either \(g(\tau)>0\) for all \(\tau\in (0, \epsilon)\), or \(g(\tau)\equiv 0\) on \((0, \epsilon)\) (the latter occurs when \(s_0\) is in the ``interior'' of a facet of \(\partial\UnitBall\)).

\medskip
A sufficient, but by no means necessary, condition ensuring that quasi-isotropy is satisfied in all directions \(s\) is that the unit ball \(\UnitBall\) has a \(C^2\) boundary with everywhere positive curvature. Other examples include, for instance, all \(\ell^p\)-norms, \(1\leq p\leq\infty\). 

%
\subsection{Main results: discussion}
%

We first informally discuss our results. Precise statements can be found in Theorem~\ref{thm:main} below.

It immediately follows from~\ref{hyp:J_path_lower_bnd} that
\begin{equation}\label{eq:TrivialUpperBoundOnICL}
	\nu_s(\lambda) \leq |s| .
\end{equation}
We say that there is \emph{saturation} at \(\lambda\) in the direction \(s\) if \(\nu_s(\lambda) = |s|\).

The function \(\lambda\mapsto \nu_s(\lambda)\) is non-increasing (see~\eqref{eq:nu_monotonicity}) and \(\lim_{\lambda\searrow 0} \nu_s(\lambda) = |s|\) (see Lemma~\ref{lem:lambda_equal_zero}).
We can thus define
\[
\lambda_{\sat}(s) = \sup\setof{\lambda}{\nu_s(\lambda) = |s|}.
\]
In several cases, we will be able to prove that \(\lambda_{\sat}(s)<\lambdaqlr\).
The main question we address in the present work is whether \(\lambda_{\sat}(s) > 0\).
Note that, when \(\lambda_{\sat} \in (0,\lambdaqlr)\), the function \(\lambda\mapsto\nu_s(\lambda)\) is not analytic in \(\lambda\).

Our main result can then be stated as follows: provided that suitable subsets of \ref{hyp:G_Bounded_Pos_nu_pos}--\ref{hyp:J_path_lower_bnd} and~\ref{hyp:PsiQuasiIsotropic} hold and \(\partial\UnitBall\) is quasi-isotropic in direction \(s\in\bbS^{d-1}\),
\[
	\lambda_{\sat}(s) > 0 \quad\Leftrightarrow\quad \sum_{y\in\bbZ^d} \psi(y)e^{-\surcharge_t(y)} < \infty ,
\]
where \(t\) is an arbitrary vector dual to \(s\).

\smallskip
What happens when quasi-isotropy fails in direction \(s\) is still mostly open; a discussion can be found in Section~\ref{sec:FailureQuasiIsotropy}.

\begin{remark}
	In a sense, exponentially decaying interactions are ``critical'' regarding the presence of a mass gap regime/condensation phenomenon.
	Indeed, on the one hand, any interaction decaying slower than exponential will lead to absence of exponential decay (e.g., \(G_\lambda(0,x)\geq C_{\lambda} J_{0x}\) by~\ref{hyp:J_path_lower_bnd} in all the models considered here).
	This is a ``trivial'' failure of mass gap, as the model is not massive.
	Moreover, the behavior \(G_\lambda(0,x)\asymp J_{0x}\) at any values of \(\lambda\) was proven in some cases: see \cite{Newman+Spohn-1998} for results on the Ising model and \cite{Aoun-2020} for the Potts model.
	On the other hand, interactions decaying faster (that is, such that \(\sup_{x\in\bbZ^d} J_{0x}e^{C\norm{x}} < \infty\) for all \(C>0\)) always lead to the presence of a mass gap (finite-range type behavior).
	Changing the prefactor to exponential decay is thus akin to exploring the ``near-critical'' regime.
\end{remark}

%
\subsection{Main Theorems}\label{sec:MainTheorems}
%

We gather here the results that are proved in the remainder of the paper.
Given a norm \(|\cdot|\) and \(s\in\bbS^{d-1}\), fix a vector \(t\) dual to \(s\) and define
\begin{equation}
	\tilde{\Xi}(|\cdot|, \psi, t) = \sum_{x\in\bbZ^d\setminus \{0\}} \psi(x) e^{-\surcharge_t(x)}.
\end{equation}
Our first result provides criteria to determine whether \(\lambda_{\sat}>0\).
\begin{theorem}
	\label{thm:main}
	Suppose~\ref{hyp:G_Bounded_Pos_nu_pos},~\ref{hyp:sub_mult},~\ref{hyp:left_cont},~\ref{hyp:weak_SL},~\ref{hyp:J_path_lower_bnd} are satisfied. Let \(s\in\bbS^{d-1}\). Then,
	\begin{itemize}
		\item If there exists \(t\) dual to \(s\) with \(\tilde{\Xi}(|\cdot|, \psi, t)<\infty\), there exists \(0<\lambda_0\leq\lambdaqlr\) such that \(\nu_{s}(\lambda)=|s|\) for any \(\lambda<\lambda_0\).
		\item Assume~\ref{hyp:PsiQuasiIsotropic}. If there exists an admissible \(t\) dual to \(s\) such that \(\partial\UnitBall\) is quasi-isotropic in direction \(s\) and \(\tilde{\Xi}(|\cdot|, \psi, t)=\infty\), then \(\nu_{s}(\lambda)<|s|\) for any \(\lambda\in(0, \lambdaqlr)\).
	\end{itemize}
	In particular, when \(\tilde{\Xi}(|\cdot|, \psi, t)<\infty\) for some \(t\) dual to \(s\), there exists \(\lambda_{\sat}\in (0,\lambdaqlr]\) such that \(\nu_{s}(\lambda) =|s|\) when \(\lambda<\lambda_{\sat}\) and \(\nu_{s}(\lambda) <|s|\) when \(\lambda>\lambda_{\sat}\).
\end{theorem}

\begin{corollary}
	\label{cor:main}
	The claim in Theorem~\ref{thm:main} applies to all the models considered in this paper (that is, \(\KRW,\SAW, \Ising, \IsingPosFie, \FK, \Potts, \GFF, \XY\)).
\end{corollary}

\begin{remark}\label{rem:DirectionDepSaturation}
	Whether \(\lambda_{\sat}(s)>0\) depends in general on the direction \(s\).
	To see this, consider the case \(|\cdot|= \normIV{\cdot}\) on \(\bbZ^2\) with \(\psi(x) = \norm{x}^{-\alpha}\) with \(7/4 \geq \alpha > 3/2\).
	
	In order to determine whether \(\lambda_{\sat}(s)>0\), it will be convenient to use the more explicit criterion derived in Lemma~\ref{lem:ExplicitCond}.
	The latter relies on the local parametrization of \(\partial\UnitBall\), as described in Section~\ref{sec:QuasiIsotropy}.
	Below, we use the notation introduced in the latter section.
	In particular, \(\lambda_{\sat}(s)>0\) if and only if
	\[
		\sum_{\ell\geq 1} \psi_0(\ell) (\ell g^{-1}(1/\ell))^{d-1} < \infty ,
	\]
	where we can take \(\psi_0(\ell) = \ell^{-\alpha}\) (remember condition~\ref{hyp:PsiQuasiIsotropic}).
	
	On the one hand, let us first consider the direction \(s=(0,1)\). 
	The corresponding dual vector is \(t=s\).
	In this case, one finds that \(f(\tau) = \frac14\tau^4 + \sfO(\tau^8)\).
	We can thus take \(g(\tau) = \tau^4\).
	In particular,
	\[
		\sum_{\ell\geq 1} \psi_0(\ell) (\ell g^{-1}(1/\ell))^{d-1}
		= \sum_{\ell\geq 1} \ell^{3/4-\alpha}
		= \infty ,
	\]
	so that \(\lambda_{\sat}(s)=0\).
	
	On the other hand, let us consider the direction \(s'=2^{-1/2}(1,1)\).
	The dual vector is \(t'=2^{-3/4}(1,1)\).
	In this case, one finds that \(f(\tau) = 3\cdot2^{-5/4}\cdot\tau^2 + \sfO(\tau^4)\).
	We can thus take \(g(\tau) = \tau^2\).
	In particular,
	\[
		\sum_{\ell\geq 1} \psi_0(\ell) (\ell g^{-1}(1/\ell))^{d-1}
		= \sum_{\ell\geq 1} \ell^{1/2-\alpha}
		< \infty ,
	\]
	so that \(\lambda_{\sat}(s)>0\).	
\end{remark}

The next theorem lists some cases in which we were able to establish the inequality \(\lambda_{\sat}< \lambdaqlr\).
\begin{theorem}\label{thm:NotSatCloseToLambdaC}
	The inequality \(\lambda_{\sat}^*<\lambda_{\exp}^*\) holds whenever one of the following is true:
	\begin{itemize}
		\item \(d=1\) and \(*\in\{\Ising,\ \FK,\ \Potts,\ \GFF,\ \XY,\ \KRW\} \);
		\item \(d\geq 2\), \(*\in \{\Ising,\Bern\}\) and \(\lambdac^{*}=\lambdaqlr^{*}\);
		\item \(d\geq 3\), \(*\in\{\GFF,\KRW\}\) and \(\lambdac^{*}=\lambdaqlr^{*}\).
	\end{itemize}
\end{theorem}
Finally, the next theorem establishes a form of condensation in part of the saturation regime.
\begin{theorem}
	\label{thm:prefactor}
	Suppose \(*\in\{\SAW,\ \Ising,\ \IsingPosFie,\ \FK,\ \Potts,\ \GFF,\ \XY\} \). Suppose moreover that \(\psi\) is one of the following:
	\begin{itemize}
		\item \(\psi(x) \propto \vert x\vert^{-\alpha}\), \(\alpha> 0\),
		\item \(\psi(x) \propto e^{-a\vert x\vert^{\alpha}}\), \(a> 0, 0<\alpha <1\).
	\end{itemize}
	Then, if \(s\in\bbS^{d-1}\) is such that \(\tilde{\Xi}(|\cdot|, \psi, t) <\infty\) for some \(t\) dual to \(s\), there exists \(\lambda_1>0\) such that, for any \(\lambda <\lambda_{1}\), there exist \(c_{\pm}=c_{\pm}(\lambda)>0\) such that 
	\begin{equation}
		c_-(\lambda) J_{0,ns}\leq G_{\lambda}^*(0,ns) \leq c_+(\lambda) J_{0,ns}.
	\end{equation}
\end{theorem}

%
\subsection{``Proof'' of Theorem~\ref{thm:main}: organization of the paper}
%

We collect here all pieces leading to the proof of Theorem~\ref{thm:main} and its corollary.
First, we have that any model \(*\in\{\SAW,\ \Ising,\ \IsingPosFie,\ \FK,\ \Potts,\ \GFF,\ \XY\}\) satisfies~\ref{hyp:G_Bounded_Pos_nu_pos},~\ref{hyp:sub_mult},~\ref{hyp:left_cont},~\ref{hyp:weak_SL}, and~\ref{hyp:J_path_lower_bnd} (see Appendix~\ref{app:Properties}).
We omit the explicit model dependence from the notation.
We therefore obtain from Claims~\ref{claim:nu_exists},~\ref{claim:nu_monotone}, and~\ref{claim:nu_trivialUB} and Lemma~\ref{lem:lambda_equal_zero} that, for any \(s\in\bbS^{d-1}\),
\begin{itemize}
	\item \(\nu_s(\lambda)\) is well defined for \(\lambda\in[0, \lambdac)\),
	\item \(\lambda\mapsto \nu_s(\lambda)\) is non-increasing,
	\item \(\lim_{\lambda\searrow 0} \nu_s(\lambda) = |s|\).
\end{itemize}
In particular, setting
\begin{equation}
	\lambda_{\sat} = \lambda_{\sat}(s) = \sup\setof{\lambda\geq 0}{\nu_s(\lambda) =|s|},
\end{equation}
it follows from monotonicity that
\begin{itemize}
	\item for any \(\lambda \in (0, \lambda_{\sat})\), \(\nu_s(\lambda) = |s|\),
	\item for any \(\lambda \in (\lambda_{\sat}, \lambdaqlr)\), \(0<\nu_s(\lambda) < |s|\).
\end{itemize}
Via a comparison with the KRW given by~\ref{hyp:weak_SL}, Lemmas~\ref{lem:SaturationKRW}, and~\ref{lem:SaturationAtSmallLambda} establish that
\begin{equation}
	\tilde{\Xi}(|\cdot|, \psi, t)<\infty \implies \lambda_{\sat}(s)>0,
\end{equation}
while Lemma~\ref{lem:mass_gap_non_summable_surcharge} implies that, when $\psi$ satisfies~\ref{hyp:PsiQuasiIsotropic} and \(\partial\UnitBall\) is quasi-isotropic in direction \(s\) (with an admissible \(t\)),
\begin{equation}
	\tilde{\Xi}(|\cdot|, \psi, t)=\infty \implies \lambda_{\sat}(s)=0,
\end{equation}
via a comparison with a suitable SAW model, allowed by~\ref{hyp:J_path_lower_bnd}.

These results are complemented in Section~\ref{sec:lambda_sat_less_lambda_c} by the inequality \(\lambda_{\sat} < \lambdaqlr\) for some particular cases (as stated in Theorem~\ref{thm:NotSatCloseToLambdaC}), using ``continuity'' properties of the models \emph{at} \(\lambdac\) and the conjectured equality \(\lambdac=\lambdaqlr\). 
Whether \(\lambda_{\sat} < \lambdaqlr\) always holds or not is an open problem (see Section~\ref{sec:open_problems}).

A proof that a condensation phenomenon (Theorem~\ref{thm:prefactor}) indeed occurs is presented in Section~\ref{sec:pre_factor}. It is carried out for a more restricted family of \(\psi\) than our main saturation result and only proves condensation in a restricted regime (see Section~\ref{sec:open_problems} for more details).

%
\subsection{Open problems and conjectures}\label{sec:open_problems}
%

The issues raised in the present work leave a number of interesting avenues open. We list some of them here, but defer the discussion of the issues related to quasi-isotropy to the next section.

\subsubsection{Is \(\lambda_{\sat}\) always smaller than \(\lambdaqlr\)?}

While this work provides precise criteria to decide whether \(\lambda_{\sat}(s)>0\), we were only able to obtain an upper bound in a limited number of cases. It would in particular be very interesting to determine whether it is possible that \(\lambda_{\sat}\) coincides with \(\lambdaqlr\), that is, that the correlation length \emph{remains constant in the whole high-temperature regime}. Let us summarize that in the following
\begin{open}
	Is it always the case that \(\lambda_{\sat}(s) < \lambdaqlr\)?
\end{open}
One model from which insight might be gained is the \(q\)-state Potts model with large \(q\). In particular, one might try to analyze the behavior of \(\nu_s(\lambda)\) for very large values of \(q\), using the perturbative tools available in this regime.

\subsubsection{What can be said about the regularity of \(\lambda\mapsto\nu_s(\lambda)\)?}

In several cases, we have established that, under suitable conditions, \(\lambdaqlr > \lambda_{\sat}(s) > 0\). In particular, this implies that \(\nu_s\) is not analytic in \(\lambda\) \emph{at} \(\lambda_{\sat}(s)\). We believe however that this is the only point at which \(\nu_s\) fails to be analytic in \(\lambda\).

\begin{conjecture}
	The inverse correlation length \(\nu_s\) is always an analytic function of \(\lambda\) on \((\lambda_{\sat}(s), \lambdaqlr)\).
\end{conjecture}
(Of course, the inverse correlation length is trivially analytic in \(\lambda\) on \([0,\lambda_{\sat}(s))\) when \(\lambda_{\sat}(s)>0\).)
\begin{conjecture}
	Assume that \(\lambda_{\sat}(s)>0\). Then, the inverse correlation length \(\nu_s\) is a continuous function of \(\lambda\) at \(\lambda_{\sat}(s)\).
\end{conjecture}
Once this is settled, one should ask more refined questions, including a description of the qualitative behavior of \(\nu_s(\lambda)\) close to \(\lambda_{\sat}(s)\), similarly to what was done in~\cite{Ott+Velenik-2018} in a case where a similar saturation phenomenon was analyzed in the context of a Potts model/FK percolation with a defect line.

\subsubsection{Sharp asymptotics for \(G_\lambda(0,x)\)}

As we explain in Section~\ref{sec:pre_factor}, the transition from the saturation regime \([0, \lambda_{\sat}(s))\) to the regime \((\lambda_{\sat}(s), \lambdaqlr)\) manifests itself in a change of behavior of the prefactor to the exponential decay of the 2-point function \(G_\lambda(0,ns)\). Namely, in the former regime, the prefactor is expected to always behave like \(\psi(ns)\), while in the latter regime, it should follow the usual OZ decay, that is, be of order \(n^{-(d-1)/2}\). This change is due to the failure of the mass gap condition of the Ornstein--Zernike theory when \(\lambda<\lambda_{\sat}(s)\). It would be interesting to obtain more detailed information.
\begin{conjecture}
	For all \(\lambda\in(\lambda_{\sat}(s), \lambdaqlr)\), \(G_\lambda(0,ns)\) exhibits OZ behavior: there exists \(C=C(s,\lambda) > 0\) such that
	\[
		G_\lambda(0,ns) = C n^{-(d-1)/2}\, e^{-\nu_s(\lambda) n} (1+\sfo(1)).
	\]
\end{conjecture}
This type of asymptotic behavior has only been established for finite-range interactions: see~\cite{Campanino+Ioffe+Velenik-2003} for the Ising model at \(\beta<\betac\), \cite{Campanino+Ioffe+Velenik-2008} for the Potts model (and, more generally FK percolation) at \(\beta<\betac\) and~\cite{Ott-2019} for the Ising model in a nonzero magnetic field (see also~\cite{Ott+Velenik-2019} for a review).
We shall come back to this problem in a future work. In the present paper, we only provide a proof in the simplest setting, the killed random walk (see Section~\ref{sec:pre_factorOZ}).

\smallskip
One should also be able to obtain sharp asymptotics in the saturation regime, refining the results in Section~\ref{sec:pre_factor}. Let \(t\) be a dual vector to \(s\). We conjecture the following to hold true.
\begin{conjecture}
	For all \(\lambda\in [0, \lambda_{\sat}(s))\), there exists $C(\lambda,s)>0$ such that \(G_\lambda(0,ns)\) exhibits the following behavior:
	\[
		G_\lambda(0,ns) = C(\lambda,s) \, \psi(ns)\, e^{-|s| n} (1+\sfo(1)),
	\]
\end{conjecture}
In this statement, $C(\lambda,s)$ depends also on the model considered. Similar asymptotics have been obtained for models with interactions decaying slower than exponential: see~\cite{Newman+Spohn-1998} for the Ising model and~\cite{Aoun-2020} for the \(q\)-state Potts model. In those cases, the constant $C(\lambda,s)$ is replaced by the susceptibility divided by $q$.

\medskip
Finally, the following problem remains completely open.
\begin{open}
	Determine the asymptotic behavior of \(G_\lambda(0,ns)\) at \(\lambda_{\sat}(s)\).
\end{open}

\subsubsection{Sharpness}

In its current formulation, Theorem~\ref{thm:NotSatCloseToLambdaC} partially relies on the equality between \(\lambdac\) and \(\lambdaqlr\). As already mentioned, we expect this to be true for all models considered in the present work.
\begin{conjecture}
	For all models considered in this work, \(\lambdac=\lambdaqlr\). 
\end{conjecture}
We plan to come back to this issue in a future work.

%
\subsection{Behavior when quasi-isotropy fails}\label{sec:FailureQuasiIsotropy}
%

In this section, we briefly discuss what we know about the case of a direction \(s\in\bbS^{d-1}\) in which the quasi-isotropy condition fails.
As this remains mostly an open problem, our discussion will essentially be limited to one particular example.
What remains valid more generally is discussed afterwards.

\medskip
We restrict our attention to \(d=2\).
Let us consider the norm \(|\cdot|\) whose unit ball consists of four quarter-circles of (Euclidean) radius \(\frac12\) and centers at \((\pm\frac12,\pm\frac12)\), joined by 4 straight line segments; see Fig.~\ref{fig:surcharge}, left.
(The associated Wulff shape is depicted in the same figure, middle.)

We are interested in the direction \(s=\frac1{\sqrt{5}}(2,1)\), in which \(\partial\UnitBall\) is \emph{not} quasi-isotropic.
The corresponding dual vector is \(t=(1,0)\).
The associated surcharge function \(\surcharge_t\) is plotted on Fig.~\ref{fig:surcharge}, right.
Observe how the presence of a facet with normal \(t\) in \(\partial\UnitBall\) makes the surcharge function degenerate: the surcharge associated to any increment in the cone \(\setof{(x,y)\in\bbZ^2}{0\leq x\leq \abs{y}/2}\) vanishes.
The direction \(s\) falls right at the boundary of this cone of zero-surcharge increments.
\begin{figure}
	\centering
	\includegraphics[width=4cm]{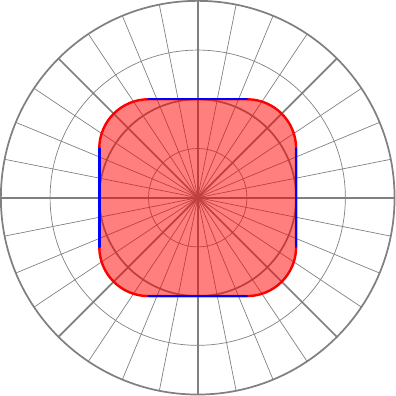}
	\hspace{1cm}
	\includegraphics[width=4cm]{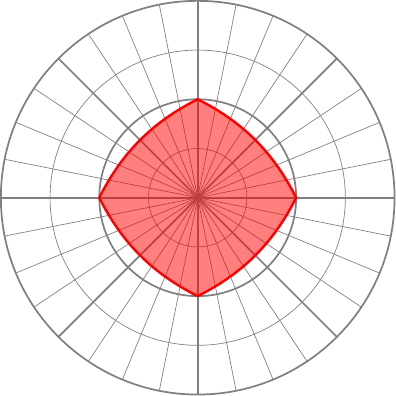}
	\hspace{1cm}
	\includegraphics[width=4cm]{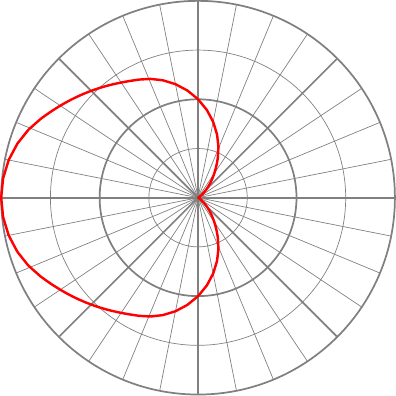}
	\caption{Left: the unit ball associated to the norm \(|\cdot|\) in the example of Section~\ref{sec:FailureQuasiIsotropy}. Middle: the corresponding Wulff shape. Right: polar plot of the surcharge function associated to the direction \(s=\frac{1}{\sqrt{5}}(2,1)\).}	
	\label{fig:surcharge}
\end{figure}

A priori, our criteria do not allow us to decide whether \(\lambda_{\sat}(s)>0\), since \(\partial\UnitBall\) (and thus the surcharge function) displays qualitatively different behaviors on each side of \(s\).
However, it turns out that, in this particular example, one can determine what is happening, using a few observations.

First, the argument in Lemma~\ref{lem:mass_gap_non_summable_surcharge} still applies provided that the sums corresponding to both halves of the cone located on each side of \(s\) diverge.
The corresponding conditions ensuring that \(\lambda_{\sat}(s)=0\) as given in~\eqref{eq:ExplicitCondition}, reduce to
\[
	\sum_{\ell\geq 1} \ell \psi_0(\ell) = \infty
\]
for the cone on the side of the facet, and
\[
	\sum_{\ell\geq 1}  \ell^{1/2} \psi_0(\ell) = \infty
\]
on the side where the curvature is positive.
Obviously, both sums diverge as soon as the second one does, while both are finite whenever the first one is.
We conclude from this that \(\lambda_{\sat}(s) > 0\) when
\[
	\sum_{\ell\geq 1} \ell \psi_0(\ell) < \infty,
\]
while \(\lambda_{\sat}(s) = 0\) when
\[
	\sum_{\ell\geq 1} \ell^{1/2} \psi_0(\ell) = \infty.
\]
Of course, this leaves undetermined the behavior when both
\begin{equation}\label{eq:ImpossibleGap}
	\sum_{\ell\geq 1} \ell \psi_0(\ell) = \infty
	\quad\text{ and }\quad
	\sum_{\ell\geq 1} \ell^{1/2} \psi_0(\ell) < \infty.
\end{equation}
However, the following simple argument allows one to determine what actually occurs in such a case.
First, observe that, since \(\nu_{s'}\leq |s'|\) for all \(s'\in\bbR^d\), the unit ball \(\UnitBall_\nu\) associated to the norm \(x\mapsto\nu_x(\lambda)\) always satisfies \(\UnitBall_\nu \supset \UnitBall\).
We now claim that this implies \(\lambda_{\sat}(s) > 0\) if and only if \(\sum_{\ell\geq 1} \ell \psi_0(\ell) < \infty\).
Indeed, suppose \(\lambda_{\sat}(s) > 0\).
Then, for small enough values of \(\lambda\), the boundaries of \(\UnitBall_\nu\) and \(\UnitBall\) coincide along the 4 circular arcs (including the points between the arcs and the facets).
But convexity of \(\UnitBall_\nu\) then implies that they must coincide everywhere, so that \(\lambda_{\sat}(s')>0\) in every direction \(s'\) pointing inside the facets.
But the latter can only occur if \(\sum_{\ell\geq 1} \ell \psi(\ell) < \infty\).
In particular, the case~\eqref{eq:ImpossibleGap} implies \(\lambda_{\sat}(s) = 0\).

\medskip
As long as we consider a two-dimensional setting, the first part of the above argument applies generally, that is, whenever quasi-isotropy fails.
The second part, however, makes crucial use of the fact that \(s\) is in the boundary of a facet of \(\partial\UnitBall\).
We don't know how to conclude the analysis when this is not the case.

\medskip
In higher dimensions, the situation is even less clear.

\begin{open}
	Provide a necessary and sufficient condition ensuring that \(\lambda_{\sat}(s)>0\) in a direction \(s\in\bbS^{d-1}\) in which \(\partial\UnitBall\) fails to be quasi-isotropic.
\end{open}


\section{Some basic properties}\label{sec:BasicProperties}


%
\subsection{Basic properties of the inverse correlation length} \label{sec:BasicPropICL}
%

A first observation is
\begin{claim}
	\label{claim:nu_exists}
	Suppose~\ref{hyp:sub_mult} holds.
	Then, \(\nu_s(\lambda)\) exists for any \(\lambda\in[0, \lambdac)\) and \(s\in\bbS^{d-1}\).
	Moreover
	\begin{equation}\label{eq:nu_infimum}
		G_{\lambda}(0,ns) \leq a_{\lambda}^{-1} e^{-\nu_s(\lambda) n}.
	\end{equation}
\end{claim}
The proof is omitted, as it is a simple variation of the classical subadditive argument.

\begin{claim}
	\label{claim:nu_norm}
	Suppose~\ref{hyp:sub_mult} holds. For \(\lambda<\lambda_{\exp}\), the function on \(\bbR^d\) defined by \(\nu_x(\lambda) = \norm{x}\cdot \nu_{x/\norm{x}}(\lambda)\) when \(x\neq 0\) and \(\nu_0(\lambda)=0\) is convex and defines a norm on \(\bbR^d\).
\end{claim}
Again, the proof is omitted, as it is a standard consequence of Assumption~\ref{hyp:sub_mult}.
Our third and fourth (trivial) observations are

\begin{claim}
	\label{claim:nu_monotone}
	Suppose~\ref{hyp:left_cont} holds.
	Then, for any \(s\in \bbS^{d-1}\), any \(x, y\in\bbZ^d\) and any \(0 \leq \lambda \leq \lambda' < \lambdac\),
	\begin{equation}\label{eq:nu_monotonicity}
		G_{\lambda}(x,y) \leq G_{\lambda'}(x,y)
		\quad\text{ and }\quad
		\nu_s(\lambda) \geq \nu_s(\lambda').
	\end{equation}
\end{claim}

\begin{claim}
	\label{claim:nu_trivialUB}
	Let \(s\in\bbS^{d-1}\). Suppose \(\nu_s(\lambda)\) is well defined and that~\ref{hyp:J_path_lower_bnd} holds.
	Then, \(\nu_s\leq |s|\).
\end{claim}

Finally, we look at the behavior of \(\nu\) when \(\lambda\searrow 0\).
\begin{lemma}
	\label{lem:lambda_equal_zero}
	Suppose~\ref{hyp:weak_SL} and~\ref{hyp:J_path_lower_bnd} hold. Then, for any \(s\in\bbS^{d-1}\), \(\lim_{\lambda\searrow 0} \nu_{s}(\lambda) = |s|\).
\end{lemma}
\begin{proof}
	Fix \(s\in\bbS^{d-1}\).
	By~\ref{hyp:J_path_lower_bnd}, \(\nu_s\leq |s|\).
	Let \(\alpha\) be given by~\ref{hyp:weak_SL}.
	Fix any \(\epsilon>0\).
	Then, let \(\lambda<\bigl( \alpha\sum_{y\neq 0} \psi(y)e^{-\epsilon|y|} \bigr)^{-1}\).
	We claim that \(G_{\lambda}(0,ns)\leq c(\lambda,\epsilon) e^{-(1-\epsilon)n|s|}\) which gives the desired claim.
	Indeed,
	\begin{align*}
		G_{\lambda}(0,ns)
		&\leq
		CG_{\alpha\lambda}^{\KRW}(0,ns) \\
		&=
		C\sum_{k\geq 1} \sum_{\substack{y_1,\dots,y_k\neq 0\\ \sum y_i=ns}} \prod_{i=1}^{k} \alpha\lambda \psi(y_i)e^{-|y_i|}\\
		&\leq
		Ce^{-(1-\epsilon)n|s|}\sum_{k\geq 1} \sum_{\substack{y_1,\dots,y_k\neq 0\\ \sum y_i=ns}} \prod_{i=1}^{k} \alpha\lambda \psi(y_i)e^{-\epsilon|y_i|}\\
		&\leq
		Ce^{-(1-\epsilon)n|s|}\sum_{k\geq 1} \Big(\lambda \sum_{y\neq 0} \alpha \psi(y)e^{-\epsilon|y|}\Big)^{\!k}.
		\qedhere
	\end{align*}
\end{proof}

%
\subsection{Weak equivalence of directions}
%

Let us introduce
\begin{equation}
	\nu_+(\lambda) = \max_{s\in\bbS^{d-1}} \nu_{s}(\lambda)
	\quad\text{ and }\quad
	\nu_-(\lambda) = \min_{s\in\bbS^{d-1}} \nu_{s}(\lambda) .
\end{equation}
The existence of these quantities follows from the fact that \(s\mapsto \nu_s(\lambda)\) is continuous (indeed, it is the restriction of a norm on \(\bbR^d\) to the set \(\bbS^{d-1}\)).

\begin{lemma}
	\label{lem:rate_equiv_directions}
	Suppose~\ref{hyp:sub_mult} holds.
	Then, \(d\cdot\nu_-(\lambda)\geq \nu_+(\lambda) \geq \nu_-(\lambda)\).
\end{lemma}
\begin{proof}
	The second inequality holds by definition.
	To obtain the first one, set \(s^*\) to be a direction realizing the minimum.
	By lattice symmetries, all its \(\pi/2\) rotations around a coordinate axis also achieve the minimum.
	For a fixed direction \(s\), denote by \(s^*_1, \dots, s^*_d\) a basis of \(\bbR^d\) constituted of rotated versions of \(s^*\) such that \(s = \sum_{i=1}^d \alpha_i s^*_i\) with \(1 \geq \alpha_i \geq 0\).
	Then, for any \(n\), \(n s = \sum_{i=1}^d n\alpha_i s_i^*\).
	So (integer parts are implicitly taken), by~\ref{hyp:sub_mult},
	\begin{equation}
		-\log G_{\lambda}(0,ns)
		\leq
		- \sum_{i=1}^d \log G_{\lambda}(0, n\alpha_i s_i^*) - d\log(a_{\lambda})
		= \sum_{i=1}^d n\alpha_i\nu_-(1+\sfo_n(1)).
	\end{equation}
	In particular, \(\lim_{n\to\infty} -\log G_{\lambda}(0, ns)/n\leq d\cdot\nu_-\).
\end{proof}

%
\subsection{Left-continuity of \(\lambda\mapsto\nu_{s}(\lambda)\)}
%

\begin{lemma}
	\label{lem:nu_left_cont}
	Suppose~\ref{hyp:sub_mult} and~\ref{hyp:left_cont} hold. Let \(s\in \bbS^{d-1}\). Let \(\lambda'\in (0,\lambdac]\) be such that
	\begin{itemize}
		\item \(G_{\lambda'}\) is well defined.
		\item There exists \(\delta>0\) such that \(\inf_{\lambda\in(\lambda'-\delta,\lambda']}a_{\lambda}>0\) (where \(a_{\lambda}\) is given by~\ref{hyp:sub_mult}).
	\end{itemize}
	Then, the function \(\lambda\mapsto \nu_s(\lambda)\) is left-continuous at \(\lambda'\).
\end{lemma}
\begin{proof}
	Fix \(\lambda'\in (0,\lambdac]\) such that \(G_{\lambda'}\) is well defined and \(s\in\bbS^{d-1}\). Let \(\delta\) be given by our hypotheses and let \(I=(\lambda'-\delta,\lambda']\), and \(C=-\log(\inf_{\lambda\in I}a_{\lambda})\). Set
	\begin{equation*}
		f_n(\lambda) = -\log G_{\lambda}(0,ns).
	\end{equation*}
	Then, for any \(\lambda\in I\) and \(n,m\in\bbZ_{>0}\), \(f_{n+m}(\lambda)\leq f_{n}(\lambda)+f_{m}(\lambda)+C\). In particular, for any \(n\geq 1\) and any \(\lambda\in I\),
	\begin{equation*}
		\nu_s(\lambda)=\lim_{q\to\infty} \frac{f_{qn}(\lambda)}{qn} \leq \frac{f_n(\lambda)}{n} + \frac{C}{n}.
	\end{equation*}
	Fix \(\epsilon>0\). Choose \(n_0\) such that \(C/n_0<\epsilon/3\) and \(\abs{\frac{f_{n_0}(\lambda')}{n_0}-\nu_s(\lambda')}\leq \epsilon/3\). By left-continuity of \(G_{\lambda}(0,n_0s)\) at \(\lambda'\), one can choose \(\epsilon'_0>0\) such that
	\begin{equation*}
		\abs{\frac{f_{n_0}(\lambda'-\epsilon')}{n_0} - \frac{f_{n_0}(\lambda')}{n_0}}\leq \epsilon/3
	\end{equation*}for any \(\epsilon'<\epsilon'_0\). In particular, for any \(\epsilon'<\epsilon'_0\),
	\begin{align*}
		0\leq \nu_{s}(\lambda'-\epsilon')-\nu_s(\lambda')
		&
		\leq \frac{f_{n_0}(\lambda'-\epsilon')}{n_0} + \frac{C}{n_0} - \nu_s(\lambda')\\
		&
		\leq  \abs{\frac{f_{n_0}(\lambda'-\epsilon')}{n_0} - \frac{f_{n_0}(\lambda')}{n_0}} + \epsilon/3 + \abs{\frac{f_{n_0}(\lambda')}{n_0} - \nu_s(\lambda')}\\
		&
		\leq \epsilon,
	\end{align*}
	where we used~\eqref{eq:nu_monotonicity} in the first line.
\end{proof}


\section{``Summable'' case}


In this section, we consider directions \(s\in\bbS^{d-1}\) for which
\begin{equation}\label{eq:SummabilityCondition}
	\sum_{y\neq 0} \psi(y) e^{-\surcharge_t(y)} < \infty,
\end{equation}
where \(t\) is any vector dual to \(s\).
In this case, we first prove that saturation occurs in direction \(s\) at small enough values of \(\lambda\), whenever the model at hand satisfies~\ref{hyp:weak_SL}.
Then, we complement this result by showing, in some models, that saturation does not occur for values of \(\lambda\) close enough to \(\lambdaqlr\).

%
\subsection{Saturation at small \(\lambda\)}
%

\begin{lemma}\label{lem:SaturationKRW}
	Let \(s\in\bbS^{d-1}\) and fix some vector \(t\) dual to \(s\). Assume that~\eqref{eq:SummabilityCondition} holds.
	Then, one can define \(0<\tilde{\lambda}\equiv \tilde{\lambda}^{\KRW}\leq \lambdac\) (given by~\eqref{eq:lambda_tilde_KRW}) such that, for any \(\lambda \in (0, \tilde{\lambda})\), \(\nu_s^{\KRW}(\lambda) = |s|\).
	Moreover, when \(d=1\), \(\tilde{\lambda}^{\KRW} = \lambda_{\sat}^{\KRW}\).
\end{lemma}
\begin{proof}
	Fix \(s\in\bbS^{d-1}\) and a dual vector \(t\). Assume that~\eqref{eq:SummabilityCondition} holds.
	Let \(G_{\lambda} \equiv G^{\KRW}_{\lambda}\).
	Set
	\begin{equation}\label{eq:lambda_tilde_KRW}
		\tilde{\lambda} = \min\Bigl\{\Bigl( \sum_{y\neq 0} \psi(y)e^{-\surcharge_t(y)}\Bigr)^{-1}, 1 \Bigr\} > 0.
	\end{equation}
	(Recall that \(\lambdac=1\) for the KRW.) Suppose \(\lambda< \tilde{\lambda}\). Let us introduce
	\begin{align*}
		A_k(n)
		&=
		\sum_{\substack{y_1, \dots, y_k\in\mathbb{Z}^d \setminus \{0\} \\ \sum_{i=1}^k y_i= ns }}\prod_{i=1}^k \lambda J_{y_i} \\
		&=
		e^{-n|s|} \sum_{\substack{y_1, \dots, y_k\in\mathbb{Z}^d \setminus \{0\} \\ \sum_{i=1}^k y_i= ns }} \prod_{i=1}^k \psi(y_i)e^{-\surcharge_t(y_i)}
		\leq e^{-n|s|} \Bigl(\lambda \sum_{y\neq 0} \psi(y) e^{-\surcharge_t(y)} \Bigr)^{\!k}.
	\end{align*}	
	Since \(\lambda \sum_{y\neq 0} \psi(y)e^{-\surcharge_t(y)} < 1\) for all \(\lambda\in [0, \tilde{\lambda})\), the first part of the result follows from
	\begin{equation}\label{eq:UB_A_n}
		G_{\lambda}(0,ns)= \sum_{k=1}^{\infty} A_k(n),
	\end{equation}
	which is a decomposition according to the length of the walk.
	
	To get the second part of the \(d=1\) case, one can assume \(\tilde{\lambda}<1=\lambdac\) (the claim being empty otherwise).
	Without loss of generality, we consider \(s=1\).
	The unique dual vector is \(t = |1|\).
	Let \(\lambda \in (\tilde{\lambda}, \lambdac)\).
	As \(\lambda < \lambdac\), \(\nu_{1}(\lambda)\) is the radius of convergence of \(\bbG_{\lambda}(z) = \sum_{n\geq 1} e^{zn} G_{\lambda}(0,n)\).
	It is therefore sufficient to find \(\epsilon>0\) such that \({\bbG_{\lambda}((1-\epsilon)|1|)} =\infty\).
	The summability of \(\bbG_{\lambda}((1-\epsilon)|1|)\) is equivalent to the summability of
	\begin{align*}
		\sum_{n\geq 1} e^{(1-\epsilon)|1|n} G_{\lambda}(0,n)
		&=
		\sum_{n\geq 1} e^{(1-\epsilon)tn} \sum_{k\geq 1} \sum_{\substack{y_1, \dots, y_k\in\mathbb{Z} \setminus \{0\} \\ \sum_{i=1}^k y_i= n }} \prod_{i=1}^k \lambda \psi(y_i) e^{-|y_i|}\\
		&=
		\sum_{k\geq 1} \sum_{y_1, \dots, y_k\in\mathbb{Z} \setminus \{0\}}\prod_{i=1}^k \lambda \psi(y_i) e^{-|y_i| +(1-\epsilon) t y_i}\\
		&=
		\sum_{k\geq 1} \Bigl( \lambda\sum_{y\in\mathbb{Z} \setminus \{0\}} \psi(y_i) e^{-\surcharge_t(y)} e^{-\epsilon |1| y} \Bigr)^{\!k}.
	\end{align*}
	Now, \(f(\epsilon) = \lambda\sum_{y\in\mathbb{Z} \setminus \{0\}} \psi(y) e^{-\surcharge_t(y)} e^{-\epsilon |1|y}\) is continuous in \(\epsilon\) on \([0,\infty)\), and \(f(0)>1\) by choice of \(\lambda\).
	So, it is still \(>1\) for some \(\epsilon>0\), implying the claim.
\end{proof}

\begin{remark}
	The statement of Lemma~\ref{lem:SaturationKRW} obviously extends to the Gaussian Free Field via~\eqref{eq:GFF_to_KRW}.
\end{remark}

We can now push the result to other models.
\begin{lemma}\label{lem:SaturationAtSmallLambda}
	Suppose~\ref{hyp:weak_SL} holds.
	Let \(s\in\bbS^{d-1}\) and \(t\) dual to \(s\). Assume that~\eqref{eq:SummabilityCondition} holds.
	Then, there exists \(\tilde\lambda>0\) such that, for any \(\lambda \in [0, \tilde\lambda)\), \(\nu_s(\lambda) = |s|\).
\end{lemma}
\begin{proof}
	Let \(\alpha\) be given by~\ref{hyp:weak_SL}.
	Set
	\begin{equation*}
		\tilde{\lambda} = \frac{1}{\alpha} \tilde{\lambda}^{\KRW} > 0.
	\end{equation*}
	By~\ref{hyp:weak_SL} and Lemma~\ref{lem:SaturationKRW}, for \(\lambda<\tilde{\lambda}'\),
	\begin{equation*}
		G_{\lambda}(0,ns) \leq C G_{\alpha\lambda}^{\KRW}(0,ns) \leq ce^{-n|s|}
	\end{equation*}
	for some \(\lambda\)-dependent constant \(c\), as \(\alpha\lambda < \tilde{\lambda}^{\KRW}\).
\end{proof}

%
\subsection{Prefactor for \(\KRW\) when \(\lambda<\lambda_{\sat}\)}\label{sec:pre_factor}
%

We first show the condensation phenomenon mentioned in the introduction for polynomial prefactors.
Namely, we prove

\begin{lemma}\label{lem:pre_fact_polynomial}
	Let \(s\in\bbS^{d-1}\) and \(t\) dual to \(s\). 
	Suppose that \(\psi(x)=C_{\alpha}\vert x\vert^{-\alpha}\) and that~\eqref{eq:SummabilityCondition} holds. Then, there exists \(\tilde{\lambda}>0\) (the same as in Lemma~\ref{lem:SaturationAtSmallLambda}) such that, for any \(\lambda<\tilde{\lambda}\), there exists \(c_+=c_{+}(\lambda)>0\) such that
	\begin{equation}
		G^{\KRW}_{\lambda}(0,ns) \leq c_+ J_{0,ns}.
	\end{equation}
\end{lemma}
\begin{remark}
	As \(\surcharge_t \geq 0\), \(\alpha>d\) always implies~\eqref{eq:SummabilityCondition}.
\end{remark}
\begin{proof}
	Fix \(s\in\bbS^{d-1}\) and a dual vector \(t\).
	Denote \(G_{\lambda}\equiv G^{\KRW}_{\lambda}\).
	Let \(\tilde{\lambda}\) be given by~\eqref{eq:lambda_tilde_KRW} and fix \(\lambda<\tilde{\lambda}\). Start as in the proof of Lemma~\ref{lem:SaturationKRW}.
	Define
	\begin{equation*}
		A_k(n)
		=
		\sum_{\substack{\gamma\in\walk(0,ns)\\ \abs{\gamma} = k}} \prod_{i=1}^k \lambda J_{\gamma_{i-1}\gamma_i}
		=
		e^{-n|s|} \sum_{\substack{y_1,\dots,y_k\neq 0\\ \sum y_i = ns}} \prod_{i=1}^k \lambda \psi(y_i) e^{-\surcharge_t(y_i)}
		\leq
		e^{-n|s|} (\lambda\tilde{\lambda}^{-1})^k.
	\end{equation*}
	Since \(\lambda < \tilde{\lambda}\), the inequality above implies that there exist \(C_{1},C_{2}>0\) such that 
	\begin{equation*}
		\sum\limits_{k=C_{1}\log(n)}^{\infty}\sum_{\substack{\gamma\in\walk(0,ns) \\ 	\abs{\gamma}=k}}\lambda^{k}\prod_{i=1}^{k}J_{\gamma_{i-1} \gamma_i}\leq C_{2}J_{0,ns}.
	\end{equation*}
	Therefore, we can assume that \(k\leq C_{1}\log(n)\).
	Let \(\gamma\in\walk(0,ns)\) with \(\vert\gamma\vert=k\).
	Since \(k<n\), there exists \(j\) such that \(\vert\gamma_{j}-\gamma_{j-1}\vert\geq \vert ns\vert /k\).
	Then, we can write
	\begin{align*}
		A_k(n)
		&\leq
		k\sum_{y: \vert y\vert\geq \vert ns\vert /k} \psi(y)e^{-\vert y\vert }\sum_{\substack{\gamma\in\walk(0,ns-y) \\ 	\abs{\gamma}=k-1}} \lambda^{k}\prod_{i=1}^{k-1}J_{\gamma_{i-1} \gamma_i}\\
		&\leq
		k e^{-n|s|}\psi(ns/k) \lambda \sum_{\substack{y_1,\dots y_{k-1}\\ \vert\sum y_i -ns\vert\geq \vert ns\vert /k } } 	\prod_{i=1}^{k-1} \lambda\psi(y_i)e^{-\surcharge_t(y_i)}\\
		&\leq
		C_3k^{1+\alpha} e^{-n|s|}\psi(ns) \lambda \Big(\sum_{y_1\neq 0} 	\lambda\psi(y_1)e^{-\surcharge_t(y_1)}\Big)^{k-1}\\
		&= C_3J_{0,ns} k^{1+\alpha} \tilde{\lambda} (\lambda \tilde{\lambda}^{-1})^{k},
	\end{align*}
	where we used \(\vert y\vert\geq \vert ns\vert /k\) and \(\surcharge_t\geq 0\) in the second line, the polynomial form of \(\psi\) in the third one, and the definition of \(\tilde{\lambda}\) in the last one. \(C_3\) is a constant depending on \(|\ |\) and \(\alpha\) only.
	This yields
	\begin{align*}
		\sum_{k=1}^{C_{1}\log(n)} A_k(n)
		\leq
		C_3J_{0,ns} \tilde{\lambda} \sum_{k=1}^{\infty}k^{\alpha+1} (\lambda \tilde{\lambda}^{-1})^k.
	\end{align*}
	Since \(\lambda<\tilde{\lambda}\), the last sum converges, which concludes the proof.
\end{proof}

We now show the same condensation phenomenon for a class of fast decaying prefactors in a perturbative regime of \(\lambda\). Namely, we assume that the function \(\psi\) satisfies
\begin{enumerate}[label={\ensuremath{\mathrm{[H_\arabic*]}}}, start=1]
	\item \label{hyp:psi_hyp1}
	\(\psi(y)\) depends only on \(\abs{y}\) and is decreasing in \(\abs{y}\).
	\item \label{hyp:psi_hyp2} there exist \(c>0\) and \(0<a\leq 1\) such that
	\begin{equation}\label{eq:prefactor_super_summability}
	\sum_{y\neq 0} \psi(y)^{a} e^{-\surcharge_t(y)} < \infty,
	\end{equation}
	and, for every \(n, m\in\bbR_+\) with \(m\leq n\),
	\begin{equation}\label{eq:prefactor_factor_bnd}
	\psi(n)\psi(m)\leq c\psi(n+m)\psi(m)^{a}.
	\end{equation}
\end{enumerate}

These assumptions are in particular true for prefactors exhibiting stretched exponential decay, \(\psi (x)= C\exp (-b \abs{x}^{\gamma} )\) with \(b>0\) and \(0<\gamma <1\), as well as for power-law decaying prefactors \(\psi(x)= C\abs{x}^{-\alpha}\) with \(\alpha>d\).

\begin{lemma}
	\label{lem:pre_fact_fast_dec}
	Fix \(s\in\bbS^{d-1}\) and a dual vector \(t\). Assume that \(\psi\) is such that~\ref{hyp:psi_hyp1} and~\ref{hyp:psi_hyp2} hold (in particular, \eqref{eq:SummabilityCondition} holds for \(t\)).
	Then, there exists \(\lambda_{0}>0\) such that, for any \(\lambda < \lambda_{0}\), one can find \(c_+>0\) such that
	\begin{equation}
		G^{\KRW}_{\lambda}(0,ns) \leq c_+ J_{0,ns} .
	\end{equation}
\end{lemma}
\begin{remark}
	On can notice that in the case \(\psi(x) = C_{\alpha}|x|^{-\alpha}\),~\eqref{eq:prefactor_factor_bnd} is satisfied with \(a=1\). In which case, \(c= 2^{\alpha}\) and~\eqref{eq:prefactor_super_summability} is simply~\eqref{eq:SummabilityCondition}. The condition is therefore the same as the one of Lemma~\ref{lem:pre_fact_polynomial} but the \(\lambda_0\) of Lemma~\ref{lem:pre_fact_fast_dec} is smaller than the \(\tilde{\lambda}\) of Lemma~\ref{lem:pre_fact_polynomial} (\(\tilde{\lambda} = 2^{\alpha} \lambda_0\)).
\end{remark}
\begin{proof}
	Fix \(s\in\bbS^{d-1}\) and a dual vector \(t\) and let \(\psi\) be as in the statement. Write \(G_{\lambda}\equiv G^{\KRW}_{\lambda}\).
	Let \(c, a\) be given by~\ref{hyp:psi_hyp2}.
	Let \(\lambda_0\) be given by
	\begin{equation*}
		\lambda_0 = \Bigl(c\sum_{y\neq 0} \psi(y)^{a}e^{-\surcharge_t(y)} \Bigr)^{\!-1}>0.
	\end{equation*}
	We can rewrite \(G_{\lambda}\) as
	\begin{align*}
		e^{n|s|} G_{\lambda}(0,ns)
		&=
		\sum\limits_{k=1}^{\infty} \lambda^k \sum_{\substack{y_1, \dots, y_k \\ \sum_{i=1}^k y_i=ns }} \prod_{i=1}^k \psi(y_i) e^{-\surcharge_t(y_i)}\\
		&\leq
		\sum_{k=1}^\infty \lambda^k k\sum_{\substack{y_1, \dots, y_{k-1}\\ \abs{ns -\sum_{i=1}^{k-1} y_i} \geq \max_i\abs{y_i}}} \psi\Bigl(ns - \sum_{i=1}^{k-1} y_i\Bigr) \prod_{i=1}^{k-1} \psi(y_i) e^{-\surcharge_t(y_i)},
	\end{align*}
	where we used \(\surcharge_t \geq 0\).
	Now, iterating~\eqref{eq:prefactor_factor_bnd} \(k\) times yields that, for any \(k\geq 1\) and any \(y_1, \dots, y_{k-1}\neq 0\) such that \(\abs{ns - \sum_{i=1}^{k-1}y_i} \geq \max_i\abs{y_i}\),
	\begin{equation*}
		\psi\Bigl(ns-\sum_{i=1}^{k-1} y_i\Bigr) \prod_{i=1}^{k-1} \psi(y_i)
		\leq
		c^k \psi(ns) \prod_{i=1}^{k-1} \psi(y_i)^{a}.
	\end{equation*}
	This gives
	\begin{equation*}
		G_{\lambda}(0,ns)
		\leq
		e^{-n|s|} \psi(ns) \lambda c \sum_{k=1}^\infty k \Bigl(\lambda c\sum_{y\neq 0} \psi(y)^{a} e^{-\surcharge_t(y)} \Bigr)^{\!k-1}.
	\end{equation*}
	The result follows since \(\lambda<\lambda_0\).
\end{proof}
As for the saturation result, one can use~\ref{hyp:weak_SL} to push the result to other models.
\begin{corollary}
	\label{cor:condensation}
	Assume that~\ref{hyp:weak_SL} and~\ref{hyp:J_path_lower_bnd} hold.
	Let \(s\in\bbS^{d-1}\) and \(t\) be a dual vector. Suppose that \(\psi\) fulfill the hypotheses of either Lemma~\ref{lem:pre_fact_polynomial} or Lemma~\ref{lem:pre_fact_fast_dec}.
	Then, there exists \(\lambda_0>0\) such that, for any \(\lambda<\lambda_0\),
	\begin{equation*}
		c_-(\lambda)J_{0,ns}\leq G_{\lambda}(0,ns) \leq c_+(\lambda) J_{0,ns},
	\end{equation*}
	for some \(c_+(\lambda),c_-(\lambda)>0\).
\end{corollary}

The use of~\ref{hyp:J_path_lower_bnd} to obtain the lower bound is obviously an overkill and the inequality follows from the less restrictive versions of the arguments we use in Appendix~\ref{app:Properties}.

%
\subsection{Prefactor for \(\KRW\) when \(\lambda>\lambda_{\sat}\)}\label{sec:pre_factorOZ}
%

In this section, we establish Ornstein--Zernike asymptotics for \(\KRW\) whenever there is a mass gap (that is, when saturation does not occur). We expect similar results for general models, but the proofs would be much more intricate. We will come back to this issue in another paper.

\begin{lemma}\label{lem:OZ}
	Let \(s\in\bbS^{d-1}\) and \(\lambda\in (\lambda_{\sat}(s),\lambdaqlr)\). There exists \(C_{\lambda}=C(\lambda)>0\) such that 
	\begin{equation}
	G_{\lambda}^{\KRW}(0,ns)=\dfrac{C_{\lambda}}{\vert ns\vert^{(d-1)/2}}e^{-\nu_{s}(\lambda)n}(1+o_{n}(1)).
	\end{equation}
\end{lemma}

\begin{proof}
We follow the ideas developed in \cite{Campanino+Ioffe-2002}.
We first express \(e^{\nu_s(\lambda) n} G^{\KRW}_{\lambda}(0,ns)\) as a sum of probabilities for a certain random walk.
We then use the usual local limit theorem on this random walk to deduce the sharp prefactor.

Let \(G_{\lambda}=G_{\lambda}^{\KRW}, \nu_s = \nu_s(\lambda)\).
Since \(\lambda<\lambdaqlr\), \(\nu\) defines a norm on \(\bbR^d\) (see Claim 2).
Let \(\tilde{t}_s\) be a dual vector to \(s\) with respect to the norm \(\nu\).
We can rewrite \(e^{\nu_s n} G_{\lambda}(0,ns)\) in the following way:
\begin{equation}
	e^{\nu_s n} G_{\lambda}(0,ns)
	=
	\sum_{N=1}^{\infty} \sum_{\substack{y_1, \dots, y_N \\ \sum y_i=ns}} \prod_{i=1}^{N} w(y_i) ,
\end{equation}
with \(w(y_i) = \lambda e^{\tilde{t}_s \cdot y_i - |y_i|} \psi(y_i)\).
Remark that \(w(y_i)\) has an exponential tail, since \(\nu_s < |s|\).
Moreover, \(w(y)\) defines a probability measure on \(\bbZ^d \setminus \{0\}\).
Indeed, let \(t_s\) be a dual vector to \(s\) with respect to the norm \(|\cdot|\).
Notice that, for \(x\in\bbR\),
\begin{align*}
	\sum_{k\geq 1} x^{|s| k}e^{\nu_s k} G_{\lambda}(0,ks)
	&= 
	\sum_{N\geq 1} \sum_{k\geq 1} \sum_{\substack{y_1,\dots,y_N \\ \sum y_i=ks }} \prod_{i=1}^{N} x^{t_s \cdot y_i} w(y_i) \\
	&\leq
	\sum_{N\geq 1}\biggl( \sum_{y\neq 0} x^{t_s \cdot y} w(y) \biggr)^{\!\!N} \\
	&=
	\dfrac{\sum_{y\neq 0} x^{t_s\cdot y} w(y)}{1-\sum_{y\neq 0} x^{t_s\cdot y} w(y)}.
\end{align*}
The radius of convergence of the series in the left-hand side is equal to 1, whereas the radius of convergence of the series in the right-hand side is strictly larger than 1, since \(w(y)\) has an exponential tail.
It follows that, for \(x=1\), we must have
\begin{equation}
\sum\limits_{y\neq 0}w(y)=1.
\end{equation}
We denote by \(P_0\) the law of the random walk \((S_n)_{n\geq 1}\) on \(\bbZ^d\), starting at \(0\in\bbZ^{d}\) and with increments of law \(w\), and by \(E_0\) the corresponding expectation.
We can rewrite
\begin{equation}\label{eq:RW}
	e^{\nu_{s}n} G_{\lambda}(0,ns) = \sum_{N\geq 1} P_0(S_N = ns). 
\end{equation}
Remark that \(E_0(S_1) = \mu s\) for some \(\mu\in\bbR\).
Indeed, were it not the case, rough large deviation bounds would imply the existence of \(c>0\) such that \(P_0(S_N = ns) \leq e^{-c\max\{n,N\}}\) for all \(N\). Using~\eqref{eq:RW}, this would imply \(e^{\nu_s n} G(0,ns) \leq e^{-c' n}\), for some \(c'>0\), contradicting the fact that \(e^{\nu_s n} G(0,ns) = e^{\sfo(n)}\).

Fix \(\delta>0\) small.
On the one hand, uniformly in \(y\) such that \(\abs{y - n\mu s} \leq n^{1/2-\delta}\), we have, by the local limit theorem,
\begin{equation}
	\sum_{N:\, \abs{N-n} \leq n^{1/2+\delta}} P_{0}(S_N=y)
	=
	\dfrac{\tilde{C_\lambda}}{\vert ns\vert^{(d-1)/2}} \bigl(1+\sfo_n(1)\bigr),
\end{equation}
where \(\tilde{C_\lambda}>0\) can be computed explicitely.
On the other hand, since \(w\) has exponential tail, a standard large deviation upper bound shows that
\begin{equation}
	\sum_{N:\,\abs{N-n} > n^{1/2 +\delta}} P_0(S_N=y) \leq e^{-c_2 n^{2\delta'}},
\end{equation}
for some small \(\delta'>0\).
Therefore, it follows from~\eqref{eq:RW} that 
\begin{equation}
	e^{\nu_s n} G_{\lambda}(0,ns) = \dfrac{C_\lambda}{|ns|^{(d-1)/2}} \bigl(1+\sfo_n(1)\bigr),
\end{equation}
with \(C_\lambda = \tilde{C_\lambda}\mu^{(d-1)/2}\).
\end{proof}

%
\subsection{Absence of saturation at large \(\lambda\)}
\label{sec:lambda_sat_less_lambda_c}
%

\begin{lemma}\label{lem:nontrivial_mass_gap_regim_d1}
	Suppose \(d=1\) and \(*\in\{\Ising, \Potts, \FK, \XY\}\).
	Then, there exists \(\lambda_0 \in (0, \infty)\) such that \(0 < \nu^*(\lambda) < \abs{1}\) when \(\lambda > \lambda_0\).
\end{lemma}
\begin{proof}
	In all the models \(\{\Ising, \Potts, \FK, \XY\}\), \(\nu(\lambda) > 0\) for any \(\lambda > 0\) when \(d=1\).
	The claim is thus an easy consequence of the finite-energy property for FK percolation: bound \(\Phi^{\FK}(0\leftrightarrow x)\) from below by the probability that a given minimal-length nearest-neighbor path \(\gamma\) is open, the probability of which is seen to be at least \(p_\beta^{\norm{x}_1}\) with \(\lim_{\beta\to\infty} p_\beta = 1\).
	A similar argument is available for the \(\XY\) model: set all coupling constants not belonging to \(\gamma\) to \(0\) by Ginibre inequalities and explicitly integrate the remaining one-dimensional nearest-neighbor model to obtain a similar bound.
\end{proof}

\begin{lemma}\label{lem:nontrivial_mass_gap_regim_GFF_KRW}
	Suppose \(*\in\{\GFF, \KRW\}\). Suppose either \(d=1\) or \(d\geq 3\) and \(\lambda^{*}_{c}=\lambda^{*}_{\exp}\). Then, \(\lambda^{*}_{\sat} < \lambda^{*}_{\exp}\).
\end{lemma}
\begin{proof}
	We treat only the KRW as extension to the GFF is immediate.
	Suppose first that \(d\geq 3\).
	Then, \(G_{\lambdac}(x,y)\) is finite for any \(x,y\in\bbZ^d\) and does not decay exponentially fast.
	So, \(\nu(\lambdac)\) is well defined and equals \(0\).
	Left continuity of \(\nu\) and the assumption \(\lambdac=\lambdaqlr\) conclude the proof.
	
	For \(d=1\) we use the characterization of Lemma~\ref{lem:SaturationKRW}.
	By our choice of normalization for \(J\) and the definition of \(\lambda_{\sat}^{\KRW}\) and \(\surcharge_t\),
	\begin{gather*}
		2\sum_{n\geq 1} \psi(n) e^{-n|1|} = 1 = \lambdac
		\quad\text{ and }\quad
		\lambda_{\sat}^{\KRW} = \Bigl( \sum_{n\geq 1} \psi(n) (1 + e^{-2n|1|}) \Bigr)^{\!-1}.
	\end{gather*}
	In particular, defining a probability measure \(p\) on \(\bbN\) by \(p(n) = 2\psi(n) e^{-n|1|}\), one obtains
	\begin{equation*}
		\lambda_{\sat}^{\KRW} = \Bigl( \sum_{n\geq 1} p(n)\cosh(n|1|) \Bigr)^{\!-1} < 1 = \lambdac^{\KRW}.
	\end{equation*}
	The conclusion will follow once we prove that $\lambda_{\exp}^{\KRW}=1$. Fix $\lambda<1$ and $\delta >0$. Then 
	\[
	\sum_{n\in\mathbb{Z}} e^{\delta n} G^{\KRW}_{\lambda}(0,n)
		=
		\sum_{n\in\mathbb{Z}} e^{\delta n} \sum_{k\geq 1} \sum_{\substack{y_1, \dots, y_k\in\mathbb{Z} \setminus \{0\} \\ \sum_{i=1}^k y_i= n }} \prod_{i=1}^k \lambda J_{0,y_{i}}
		=\sum_{k=1}^{\infty}\Bigl(\lambda\sum_{y\neq 0}J_{0,y}e^{\delta y}\Bigr)^{\!\!k}.
	\]
By our choice of normalization for $J$ and the fact that $J_{0,y}$ has exponential tails, it is possible to find $\delta$ small enough such that the sum over $k$ is finite, which proves that $\lambda_{\exp}^{\KRW}=1$.
\end{proof}

\begin{lemma}\label{lem:nontrivial_mass_gap_regim_AnyD}
	Suppose \(d>1\) and consider Bernoulli percolation or the Ising model. Suppose \(\lambdaqlr=\lambdac\). 
	Then, there exists \(\lambda_0 \in [0, \lambdaqlr)\) such that, for any \(s\in\bbS^{d-1}\) and \(\lambda \in (\lambda_0, \lambdaqlr)\),
	\begin{equation*}
		\nu_s(\lambda) < |s|.
	\end{equation*}
\end{lemma}
\begin{proof}
	The existence of \(\lambda_0\) follows from Lemma~\ref{lem:nu_left_cont} and the fact that \(\nu_s(\lambdac) = 0\) which is obtained by equivalence of directions for \(\nu\) (Lemma~\ref{lem:rate_equiv_directions}) and divergence of the susceptibility at \(\lambdac\).
	The latter is proved for the Ising model and Bernoulli percolation in~\cite{Duminil-Copin+Tassion-2016}. The conclusion follows by the assumption \(\lambdaqlr=\lambdac\).
\end{proof}


\section{``Non-summable'' case}


In this section we consider directions \(s\in\bbS^{d-1}\) for which
\begin{equation}\label{eq:NonSummabilityCondition}
	\sum_{y\neq 0} \psi(y) e^{-\surcharge_t(y)} = +\infty,
\end{equation}
where \(t\) is any vector dual to \(s\).
We prove that saturation does not occur in direction \(s\) at any value of \(\lambda\), provided that the model at hand satisfies~\ref{hyp:J_path_lower_bnd}.

\medskip
Before proving the general claim, let us just mention that the claim is almost immediate when \(\psi(ns)\) is not uniformly bounded in \(n\).
Indeed, suppose \(\nu_s(\lambda)=|s|\).
Then, by~\ref{hyp:sub_mult}, \(G_{\lambda}(0,ns)\leq a_{\lambda}^{-1}e^{-\nu_s n}\) (using~\eqref{eq:nu_infimum}), while by~\ref{hyp:J_path_lower_bnd}, \(G_{\lambda}(0,ns)\geq C_{\lambda} \psi(ns)e^{-n|s|}\).
From these two assumptions and the assumption that \(\nu_s(\lambda)=|s|\), we deduce that
\begin{equation*}
	C_{\lambda} \psi(ns)e^{-n|s|}\leq G_{\lambda}(0,ns) \leq a_{\lambda}^{-1}e^{-n|s|},
\end{equation*}
which implies that \(\psi(ns)\) is bounded uniformly over \(n\).

\medskip
Let us now turn to a proof of the general case.

%
\subsection{Absence of saturation at any \(\lambda\)}
%

\begin{lemma}\label{lem:mass_gap_non_summable_surcharge}
	Suppose~\ref{hyp:J_path_lower_bnd} and~\ref{hyp:PsiQuasiIsotropic}. Let \(s\in\bbS^{d-1}\) and let \(t\) be a vector dual to \(s\). Assume that \(\partial\UnitBall\) is quasi-isotropic in direction \(s\) and that~\eqref{eq:NonSummabilityCondition} holds.
	Then, for any \(\lambda>0\), \(\nu_s(\lambda)<|s|\).
\end{lemma}
\begin{proof}
	We use the notation of Section~\ref{sec:QuasiIsotropy}. In particular, we assume that \(\scrN\) and \(\epsilon\) have been chosen small enough to ensure that either \(g\equiv 0\), or \(g\) vanishes only at \(0\).
	
	Let \(\delta>0\) and consider the cone \(\cone_{t,\delta} = \setof{y\in\bbZ^d}{\surcharge_t(y) \leq \delta |y|}\).
	When \(g\) vanishes only at \(0\), we further assume that \(\delta\) is small enough to ensure that \(\cone_{t,\delta} \cap \partial\UnitBall \subset \scrN\) (this will be useful in the proof of Lemma~\ref{lem:ExplicitCond} below.)
	
	It follows from~\eqref{eq:psi_subexp} that
	\[
	\sum_{y\notin\cone_{t,\delta}} \psi(y) e^{-\surcharge_t(y)}
	\leq
	\sum_{y\notin\cone_{t,\delta}} \psi(y) e^{-\delta |y|} < \infty .
	\]
	Since we assume that~\eqref{eq:NonSummabilityCondition} holds, this implies that
	\[
	\sum_{y\in\cone_{t,\delta}} \psi(y) e^{-\surcharge_t(y)} = +\infty.
	\]
	Let \(\calT_R(s) = \setof{y\in\bbR^d}{\normsup{y-(y\cdot s)s} \leq R}\). We will need the following lemma.
	\begin{lemma}\label{lem:intermediaire}
	For any \(R>0\) large enough, we have 
	\begin{equation}\label{eq:DivergenceSubCone}
		\inf_{x\in\calT_R(s)} \sum_{y\in (x+\cone_{t,\delta}) \cap \calT_R(s)} \psi(y-x) e^{-\surcharge_t(y-x)} = \infty .
	\end{equation}
	\end{lemma}
	This lemma is established below.
	In the meantime, assume that the lemma is true. Then, one can find \(R>0\) such that
	\begin{equation}\label{eq:BigEnough}
		\inf_{x\in\calT_R(s)} \sum_{y\in (x+\cone^R_{t,\delta}) \cap \calT_R(s)} \psi(y-x) 	e^{-\surcharge_t(y-x)} \geq e^2 C_\lambda^{-1} .
	\end{equation}
	where we have introduced the truncated cone \(\cone^R_{t,\delta} = \setof{y\in\cone_{t,\delta}}{\normsup{s}\leq R}\).

	We are now going to construct a family of self-avoiding paths connecting \(0\) to \(ns\) in the following way: we first set \(M=\frac{n}{2R}\) and choose \(y_1, y_2, \dots, y_{M+1}\) in such a way that
	\begin{itemize}
		\item \(y_k \in \cone^R_{t,\delta}\) for all \(1\leq k\leq M\);
		\item for all \(1\leq m\leq M\), \(\sum_{k=1}^m y_k \in \calT_R(s)\);
		\item \(y_{M+1} = ns- \sum_{k=1}^M y_k\).
	\end{itemize}
	Note that, necessarily, \(s\cdot y_{M+1} \geq n/2\) and \(y_{M+1}\in \calT_R(s)\).
	We then consider  the set \(\Gamma\subset\saw(0,ns)\) of all self-avoiding paths \((0, y_1, y_1+y_2, \dots, y_1+\dots+y_{M}, ns)\) meeting the above requirements.
	
	We thus obtain that, by~\ref{hyp:J_path_lower_bnd},
	\begin{align*}
		e^{n|s|} G_{\lambda}(0,ns)
		&\geq
		C_{\lambda}\sum_{y_1}\dots\sum_{y_M} \prod_{k=1}^{M+1} C_{\lambda} \psi(y_k) e^{-|y_k| + y_k\cdot t} \\
		&=
		(C_{\lambda})^{M +2} e^{\sfo(n)} \sum_{y_1}\dots\sum_{y_M} \prod_{k=1}^{M} \psi(y_k) e^{-\surcharge_t(y_k)}\\
		&\geq
		(C_{\lambda})^{M +2} e^{\sfo(n)} \sum_{y_1}\dots\sum_{y_{M-1}} \prod_{k=1}^{M-1} \psi(y_k) e^{-\surcharge_t(y_k)} (e^2 C_{\lambda}^{-1})\\
		&\geq\cdots\geq
		(C_{\lambda})^{M +2} e^{\sfo(n)} (e^2 C_{\lambda}^{-1})^M = C_{\lambda}^2 e^{n/R +\sfo(n)},
	\end{align*}
	where the sums are over \( y_1, \ldots, y_M\) meeting the requirements for the path to be in \(\Gamma\).
	The term \(e^{\sfo(n)}\) in the second line is the contribution of \(y_{M+1}\) (\(y_{M+1}\in\calT_R(s)\) and its length is at least \(n/2\), so \(\surcharge_t(y_{M+1}) = \sfo(n)\) and \(\psi(y_{M+1})=e^{\sfo(n)}\)).
	For the third and fourth lines, we apply~\eqref{eq:BigEnough} \(M\) times.
\end{proof}	
	There only remains to prove Lemma~\ref{lem:intermediaire}.
	The latter is a direct consequence of the following quantitative version of~\eqref{eq:NonSummabilityCondition}, which can be useful to explicitly determine whether saturation occurs in a given direction; see Remark~\ref{rem:DirectionDepSaturation} in Section~\ref{sec:MainTheorems} for an example. Below, it will be convenient to set \(g^{-1}\equiv 1\) when \(g\equiv 0\).
	\begin{lemma}\label{lem:ExplicitCond}
	Under the assumptions of Lemma~\ref{lem:mass_gap_non_summable_surcharge}, Condition~\eqref{eq:NonSummabilityCondition} is equivalent to the condition
		\begin{equation}\label{eq:ExplicitCondition}
			\sum_{\ell\geq 1} \psi_0(\ell) (\ell g^{-1}(1/\ell))^{d-1} = \infty.
		\end{equation}
	\end{lemma}
	\begin{proof}
	We shall do this separately for the case \(g\equiv 0\) (\(s_0\) belongs to the ``interior'' of a facet of \(\partial\UnitBall\)) and when \(g\) vanishes only at \(0\).
	
	\medskip
	\textbf{Case 1: \(\boldsymbol{g\equiv 0}\).}
	In this case, we can find \(\eta>0\) such that \(\surcharge_t(y) = 0\) for all \(y\) in the subcone \(\scrC_\eta(s) = \setof{\lambda s'}{\lambda>0,\, s'\in\bbS^{d-1},\, \norm{s'-s} < \eta}\). 
	In particular, for all \(y\in\scrC_\eta(s)\),
	\[
		\psi(y) e^{-\surcharge_t(y)} = \psi(y),
	\]
	from which the claim follows immediately using~\ref{hyp:PsiQuasiIsotropic}.
	
	\medskip
	\textbf{Case 2: \(\boldsymbol{g > 0}\).}
	We now assume that \(g(\tau) > 0\) for all \(\tau\neq 0\) (remember the setting of Section~\ref{sec:QuasiIsotropy}).
	For simplicity, let \(u\in\bbZ^d\) be such that \(\normsup{u}=R\) and write \(\scrC_u = \cone_{t,\delta} \cap \bigl(u + \calT_R(s)\bigr)\) for the corresponding sub-cone. 
	
	Given \(y\in\cone_{t,\delta}\), we write \(y^\parallel = y\cdot \hat{t}\) and \(y^\perp = y - y^\parallel \hat{t}\). In particular, we have
	\[
		y^\parallel = \frac{|y|}{\norm{t}} - |y| f\biggl(\frac{y^\perp}{|y|}\biggr).
	\]
	This implies that
	\[
		\surcharge_t(y)
		= |y| - t\cdot y
		= |y| - \norm{t} y^\parallel
		= \norm{t} |y|\, f(y^\perp/|y|).
	\]
	We conclude that
	\begin{equation}\label{eq:boundsOnSurcharge}
		C_+ |y|\, g(\norm{y^\perp}/|y|)
		\geq
		\surcharge_t(y)
		\geq
		C_- |y|\, g(\norm{y^\perp}/|y|)
	\end{equation}
	where we have set \(C_\pm = c_\pm \norm{t}\).
	Using~\ref{hyp:PsiQuasiIsotropic}, we can write
	\begin{align*}
		\sum_{y\in\scrC_u} \psi(y) e^{-\surcharge_t(y)}
		\leq
		C_\psi^+\sum_{\ell\geq 1} \psi_0(\ell) \sum_{r\geq 0}  \sum_{\substack{y\in\scrC_u\\\normI{y}=\ell\\\norm{y^\perp}\in [r,r+1)}} 
		e^{-\surcharge_t(y)}
		\leq
		c_1 \sum_{\ell\geq 1} \psi_0(\ell) \sum_{r\geq 0}
		r^{d-2} 
		e^{-c_2\ell g(c_3r/\ell)} .
	\end{align*}
Let \(x=\frac{1}{c_3} \ell g^{-1}(1/\ell)\). The sum over \(r\) is easily bounded. : 
\begin{align*}
\sum_{r\geq 0}
		r^{d-2} 
		e^{-C_- \ell g(r/\ell)}
		&\leq
		\sum_{k\geq 0} \sum_{r=kx}^{(k+1)x} r^{d-2} e^{-c_2\ell g(c_3 r/\ell)} \\
		&\leq
		\sum_{k\geq 0} e^{-c_2 \ell g(k g^{-1}(1/\ell))} \sum_{r=kx}^{(k+1)x} r^{d-2} \\
		&\leq
		x^{d-1} \sum_{k\geq 0} (k+1)^{d-1} e^{-c_2 \ell g(k g^{-1}(1/\ell))} .
\end{align*}
Let us prove that the last sum is finite. Let \(h(k) = g(k g^{-1}(1/\ell))\). Notice that \(h(0)=g(0)=0\) and \(h(1)=1/\ell\). Since \(g\) is convex and increasing, \(h\) is convex and increasing as well. Therefore, convexity implies that
\begin{equation*}
	h(1)
	=
	h\bigl( \tfrac1k\cdot k + (1-\tfrac1k)\cdot 0 \bigr)
	\leq
	\tfrac1k h(k) + (1-\tfrac1k) h(0)
	=	
	\tfrac1k h(k) .
\end{equation*}
	Therefore, we get 
\begin{equation*}
	\sum_{k\geq 0} k^{d-1} e^{-c_2 \ell g(k g^{-1}(1/\ell))}
	\leq 
	\sum_{k\geq 0} (k+1)^{d-1} e^{-c_2 k} ,
\end{equation*}
which implies the following upper bound
\begin{equation*}
	\sum_{y\in\scrC_u} \psi(y) e^{-\surcharge_t(y)}
	\leq 
	c_4\sum_{\ell\geq 1} \psi_0(\ell) (\ell g^{-1}(1/\ell))^{d-1} .
\end{equation*}
Similarly, using the upper bound in~\eqref{eq:boundsOnSurcharge} (and once more~\ref{hyp:PsiQuasiIsotropic}), we get the following lower bound :
\begin{align*}
	\sum_{y\in\scrC_u} \psi(y) e^{-\surcharge_t(y)}
	&\geq
	C_\psi^-\sum_{\ell\geq 1} \psi_0(\ell) \sum_{r\geq 0}  \sum_{\substack{y\in\cone_{t,\delta}\\\normI{y}=\ell\\\norm{y^\perp}\in [r,r+1)}} e^{-\surcharge_t(y)} \\
	&\geq 
	C_\psi^-\sum_{\ell\geq 1} \psi_0(\ell) \sum_{r=0}^{\frac{1}{c_6}\ell g^{-1}(1/\ell)} r^{d-2} e^{-c_5\ell g(c_6 r/\ell)} \\
	&\geq 
	c_7 \sum_{\ell\geq 1} \psi_0(\ell) \sum_{r=0}^{\frac{1}{c_6}\ell g^{-1}(1/\ell)} r^{d-2} \\
	&\geq
	c_8 \sum_{\ell\geq 1} \psi_0(\ell)(\ell g^{-1}(1/\ell))^{d-1} . \qedhere
\end{align*}
\end{proof}


\section{Acknowledgments}


Dima Ioffe passed away before this paper was completed.
The first stages of this work were accomplished while he and the fourth author were stranded at Geneva airport during 31 hours. In retrospect, the fourth author is really grateful to easyJet for having given him that much additional time to spend with such a wonderful friend and collaborator.

YA thanks Hugo Duminil-Copin for financial support. SO is supported by the Swiss NSF through an early PostDoc.Mobility Grant. SO also thanks the university Roma Tre for its hospitality, hospitality supported by the ERC (ERC CoG UniCoSM, grant agreement n.724939). YV is partially supported by the Swiss NSF.


\appendix


\section{Proof of the assumptions}\label{app:Properties}


%
\subsection{Assumption~\ref{hyp:sub_mult}}
%

\begin{lemma}
	\ref{hyp:sub_mult} holds for: \(\KRW, \SAW, \Ising, \IsingPosFie, \Potts, \FK, \XY, \GFF\).
\end{lemma}
\begin{proof}
	The desired inequality follows with \(a_{\lambda} =1\) from GKS/Ginibre inequalities for the Ising and XY models, from the FKG inequality for FK percolation (and thus for the Potts model using~\eqref{eq:Potts_FK_Corresp}).
	\(\IsingPosFie\) is the main claim in~\cite{Graham-1982} (still for \(a_{\lambda} =1\)).
	For the \(\GFF\), \ref{hyp:sub_mult} holds with
	\begin{equation*}
		a_{\lambda} = G_{\lambda}^{\GFF}(0,0)^{-1}
	\end{equation*}(for \(\lambda<\lambdac\)). Indeed, using the random walk representation~\eqref{eq:GFF_RW_Rep_Cov},
	\begin{align*}
		G_{\lambda}(x,y)
		&=
		\lambda \sum_{n\geq 0} P_{x}^m(S_n= y, T>n)\\
		&\geq
		\lambda \sum_{n\geq 0}\sum_{k=0}^{n} P_{x}^m(S_n= y,T_z =k, T>n)\\
		&=
		\lambda \sum_{k\geq 0}\sum_{n\geq k} P_{x}^m(T_z =k, T>k) P_{z}^m(S_{n-k}= y, T>n-k)\\
		&=
		P_{x}^m(T_z < T) G_{\lambda}(z,y) ,
	\end{align*}
	where \(T_z = \min\setof{k\geq 0}{S_k=z}\).
	Now, \(G_{\lambda}(x,z) = P_{x}^m(T_z < T) G_{\lambda}(0,0) \), from which the claim follows.
	The identity~\eqref{eq:GFF_to_KRW} implies that the inequality holds for \(\KRW\) with \(a_{\lambda} = G_{\lambda}^{\KRW}(0,0)^{-1}\).
	For \(\SAW\), one has the inequality with
	\begin{equation*}
		a_{\lambda} = \Big(\sum_{x\in \bbZ^d}G_{\lambda}^{\SAW}(0,x)^2\Big)^{-1}.
	\end{equation*}
	Informally: from \(\gamma\in \saw(x,z)\) and \(\gamma'\in \saw(y,z)\), build \(\gamma''\in \saw(x,y)\) by following \(\gamma\) until its first intersection with \(\gamma'\), denoted \(\tau\), and by then following \(\gamma'\) backward until reaching \(y\).
	The remaining sub-paths can obviously be split into two walks in \(\saw(z,\tau)\).
	Summing over \(\gamma,\gamma'\), one obtains \(G_{\lambda}^{\SAW}(x,z)G_{\lambda}^{\SAW}(y,z)\).
	\(\gamma''\) gives the \(G_{\lambda}^{\SAW}(x,y)\) contribution while summing over \(\tau\) gives the \(\sum_{\tau\in \bbZ^d} G_{\lambda}^{\SAW}(0,\tau)^2\) contribution.
\end{proof}

%
\subsection{Assumption~\ref{hyp:left_cont}}
%

\subsubsection{KRW, SAW, GFF}

Since \(G_{\lambda}^{\KRW}(0,x)\) and \(G^{\SAW}_{\lambda}\) are power series in \(\lambda\), monotonicity is clear.
Moreover, their radius of convergence is \(\lambdac\) so that, for any \(\lambda <\lambdac\), these 2-point functions are analytic (and in particular continuous).
The result for the GFF follows trivially from~\eqref{eq:GFF_to_KRW}.

\subsubsection{Potts/Ising models and FK percolation}

We only discuss the results for FK percolation, the claims for the Ising/Potts models following immediately from~\eqref{eq:Potts_FK_Corresp}.

For any \(n\in\bbN\), \(\Phi_{\Lambda_{n},\beta,q}^{\FK} (0\leftrightarrow x)\) is differentiable with derivative equal to 
\begin{equation}
	\dfrac{\dd}{\dd\beta} \Phi^{\FK}_{\Lambda_{n},\beta,q} (0\leftrightarrow x)
	= \sum_{\{u,v\}\subset\Lambda_n} \bigl(
	\Phi^{\FK}_{\Lambda_n,\beta,q} (0\leftrightarrow x \given \omega_{uv} = 1) -
	\Phi^{\FK}_{\Lambda_n,\beta,q} (0\leftrightarrow x \given \omega_{uv} = 0)
	\bigr)
	\geq 0 ,
\end{equation}
thanks to the FKG inequality.
It follows that \(\Phi^{\FK}_{\beta,q} (0\leftrightarrow x) = \lim_{n\to\infty} \Phi_{\Lambda_n,\beta,q}^{\FK} (0\leftrightarrow x)\) is non-decreasing in \(\beta\).

Let us prove that \(\Phi^{FK}_{\beta,q} (0\leftrightarrow x)\) is left-continuous.
It follows from the FKG inequality that, for any \(n\geq m\), \(\Phi_{\Lambda_{n+1},\beta,q}^{\FK} (0\xleftrightarrow{\Lambda_m} x) \geq \Phi_{\Lambda_n,\beta,q}^{\FK} (0\xleftrightarrow{\Lambda_m} x)\), 
where \(\{0\xleftrightarrow{\Lambda_m} x\}\) is the event that \(0\) and \(x\) are connected by a path of open edges inside \(\Lambda_m\).
Therefore
\begin{equation*}
	\Phi^{\FK}_{\beta,q} (0\leftrightarrow x)
	\geq
	\Phi_{\Lambda_n,\beta,q}^{\FK} (0\xleftrightarrow{\Lambda_n} x)
	\geq
	\Phi_{\Lambda_n,\beta,q}^{\FK} (0\xleftrightarrow{\Lambda_m} x) ,
\end{equation*}
where we used inclusion of events in the last inequality.
Taking the limits \(n\to\infty\) followed by \(m\to\infty\), we conclude that the sequence \(\Phi_{\Lambda_n,\beta,q}^{\FK} (0\xleftrightarrow{\Lambda_n} x)\) converges to \(\Phi_{\beta,q}^{\FK} (0\leftrightarrow x)\).
Moreover, remark that 
\begin{equation*}
	\Phi_{\Lambda_n,\beta,q}^{\FK} (0\xleftrightarrow{\Lambda_n} x)
	\leq
	\Phi_{\Lambda_{n+1},\beta}^{\FK}(0\xleftrightarrow{\Lambda_{n}} x)
	\leq
	\Phi_{\Lambda_{n+1},\beta,q}^{\FK}(0\xleftrightarrow{\Lambda_{n+1}} x) ,
\end{equation*}
where we used monotonicity in volume in the first inequality and inclusion of events in the second inequality.
Therefore, the sequence \(\Phi_{\Lambda_n,\beta,q}^{\FK} (0\xleftrightarrow{\Lambda_n} x)\) is non-decreasing and converges to \(\Phi_{\beta,q}^{\FK} (0\leftrightarrow x)\).
Each \(\Phi_{\Lambda_n,\beta,q}^{\FK} (0\xleftrightarrow{\Lambda_{n}} x)\) is continuous in \(\beta\), whence \(\Phi_{\beta,q}^{\FK} (0\leftrightarrow x)\) is left-continuous.

\subsubsection{Ising with positive field}

Let us first prove that \(G^{\IsingPosFie}_{\beta,h}\) is non-increasing in \(h\).
Fix \(\beta, h\geq 0\) and \(n\in\bbN\).
For a subset \(A\subset\bbZ^d\), let us write \(\langle\sigma_{A}\rangle = \mu^{\Ising}_{\Lambda_n,\beta,h} (\sigma_{A})\).
Then, the function \(G^{\IsingPosFie}_{\Lambda_n;\beta,h} (0,x) = \langle\sigma_0\sigma_x\rangle - \langle\sigma_0 \rangle \langle\sigma_x\rangle\) is differentiable in \(h\) with derivative
\begin{multline*}
	\dfrac{\dd}{\dd h} G^{\IsingPosFie}_{\Lambda_n;\beta,h} (0,x)\\
	= \sum_{i\in\Lambda_n}
	\langle\sigma_0\sigma_x\sigma_i\rangle
	- \langle\sigma_0\sigma_x\rangle \langle\sigma_i\rangle
	- \langle\sigma_0\sigma_i\rangle \langle\sigma_x\rangle
	- \langle\sigma_x\sigma_i\rangle \langle\sigma_0\rangle
	+ 2\langle\sigma_0\rangle \langle\sigma_x\rangle \langle\sigma_i\rangle
	\leq 0 ,
\end{multline*}
where we used the GHS inequality~\cite{Griffiths+Hurst+Sherman-1970}.
By taking the limit \(n\to\infty\), we get that \(G^{\IsingPosFie}_{\beta,h}\) is non-increasing in \(h\) (thus non-decreasing in \(\lambda=e^{-h}\)). 

Let us now prove that \(G^{\IsingPosFie}_{\beta,h}\) is right-continuous in \(h\).
Observe that it is enough to prove that, for \(A\subset\bbZ^d\), \(\mu^+_{\beta,h} (\sigma_A)\) is right-continuous in \(h\) (see~\cite[Chapter 3]{Friedli+Velenik-2017} for the definition of \(\mu_{\beta,h}^{+}\)).
Fix \(h > 0\) and let \((h_m)_{m\geq 1}\) be a non-increasing sequence of real numbers converging to \(h\).
It follows from the GKS inequalities that, for any \(n,m\in\bbN\), \(\mu^+_{\Lambda_n;\beta,h_m} (\sigma_A)\) is non-increasing in \(n\) and that 
\begin{equation}
	\dfrac{\dd}{\dd h} \mu^+_{\Lambda_n;\beta,h_m} (\sigma_A)
	=
	\sum_{i\in\Lambda_n} \mu^+_{\Lambda_n;\beta,h_m} (\sigma_A\sigma_i) - \mu^+_{\Lambda_n;\beta,h_m} (\sigma_A) \mu^+_{\Lambda_n;\beta,h_m} (\sigma_i)
	\geq 0.
\end{equation}
Therefore, \((\mu^+_{\Lambda_n;\beta,h_m} (\sigma_A))_{m,n\geq 1}\) is non-increasing in \(n\) and in \(m\). The limits can thus be interchanged:
\begin{align*}
	\lim_{m\to\infty} \mu_{\beta,h_m}^{\Ising} (\sigma_A)
	&=
	\lim_{m\to\infty} \lim_{n\to\infty} \mu^+_{\Lambda_n;\beta,h_m}(\sigma_A) \\
	&=
	\lim_{n\to\infty} \lim_{m\to\infty} \mu^+_{\Lambda_n;\beta,h_m} (\sigma_A)
	=
	\lim_{n\to\infty} \mu^+_{\Lambda_n;\beta,h} (\sigma_A)
	=
	\mu^{\Ising}_{\beta,h}(\sigma_A),
\end{align*}
where the third identity relies on the fact that \(\mu^{+}_{\Lambda_{n},\beta,\lambda}(\sigma_{A})\) is continuous in \(h\) and the first and last ones from the uniqueness of the infinite-volume Gibbs measure in non-zero magnetic field (see~\cite{Friedli+Velenik-2017} for instance).
We conclude that \(\mu^{\Ising}_{\beta,\lambda}\) is right-continuous in \(h\) (thus left-continuous in \(\lambda=e^{-h}\)). 

\subsubsection{XY model}

Fix \(n\in\bbN\) and \(\beta\geq 0\). Let \(\theta_{x,y} = \cos(\theta_{x}-\theta_{y})\).
It follows from the Ginibre inequalities~\cite{Ginibre-1970} that 
\begin{equation}
	\dfrac{\dd}{\dd\beta } G^{\XY}_{\Lambda_n;\beta} (0,x)
	=
	\sum_{\{y,z\}\subset\Lambda_n} \bigl(
	\mu^{\XY}_{\Lambda_n;\beta} (\theta_{0,x}\theta_{y,z})
	- \mu^{\XY}_{\Lambda_n;\beta} (\theta_{0,x}) \mu^{\XY}_{\Lambda_n;\beta} (\theta_{y,z}) \bigr)
	\geq 0 .
\end{equation}
Therefore, by taking the limit in \(n\), we get that \(G^{\XY}_{\beta}(0,x)\) is non-decreasing in \(\lambda=\beta\).

Let us turn to the proof of left-continuity of \(G^{\XY}_{\Lambda_n;\beta}\) in \(\lambda=\beta\).
Observe that it is enough to prove that, for any collection \((M_i)_{i\in\bbZ^d}\) of integers such that \(M_i=0\) for all but finitely many vertices \(i\in\bbZ^d\), \(\mu^{\XY}_{\beta} (\cos(M\theta))\) is left-continuous in \(\beta\). Here, we are using the notation \(M\theta = \sum_i M_i\theta_i\).

Fix \(\beta > 0\) and let \((\beta_m)_{m\geq 1}\) be a non-decreasing sequence of real numbers converging to \(\beta\). The same argument we used for the Ising model in a field will allow us to conclude once we know that \(n\mapsto \mu^{\XY}_{\Lambda_n;\beta} (\cos(M\theta))\) and \(m\mapsto \mu^{\XY}_{\Lambda_n;\beta_m} (\cos(M\theta))\) are both non-decreasing. But this is an immediate consequence of the Ginibre inequalities~\cite{Ginibre-1970}.

%
\subsection{Assumption~\ref{hyp:weak_SL}}
%

The assumption is obviously satisfied for \(\KRW\) (and therefore \(\GFF\) by~\eqref{eq:GFF_to_KRW}) and \(\SAW\).

\subsubsection{Potts/Ising models and FK percolation}

We prove the inequality~\eqref{eq:weak_SL} for FK percolation and use~\eqref{eq:Potts_FK_Corresp} to deduce the result for the Ising/Potts models.

The inequality follows from the finite-energy property of the model and the fact that \(x\leftrightarrow z\) implies that there exists \(y\neq x\) such that
\begin{enumerate*}[label=(\roman*)]
	\item \(\omega_{xy}=1\),
	\item \(y\) is connected to \(z\) without using the edge \(\{x,y\}\).
\end{enumerate*}
Denote this event \(\{y\xleftrightarrow{\{x, y\}^{\comp}} z\}\).
It is measurable with respect to the sigma-algebra generated by \(\{\omega_e\}_{e\neq\{x,y\}}\).
By a union bound, one then has
\begin{align*}
	\Phi^{\FK}_{\beta}(x\leftrightarrow z)
	&\leq
	\sum_{y\neq x} \Phi^{\FK}_{\beta}(y \xleftrightarrow{\{x, y\}^{\comp}} z) \Phi^{\FK}_{\beta}(\omega_{xy}=1 \given y \xleftrightarrow{\{x, y\}^{\comp}} z) \\
	&\leq
	\sum_{y\neq x} \bigl(1-e^{-\beta J_{xy}}\bigr) \Phi^{\FK}_{\beta}(y \leftrightarrow z) \\
	&\leq
	\sum_{y\neq x} \beta J_{xy} \Phi^{\FK}_{\beta}(y \leftrightarrow z).
\end{align*}
Iterating until \(y\) reaches \(z\) yields the result with \(\alpha =1\) in \(\Potts,\FK\) and \(\alpha=2\) for \(\Ising\).

\subsubsection{XY model}

The inequality is proven in~\cite{Brydges+Frohlich+Spencer-1982} for a vast class of \(O(N)\)-symmetric models with \(C=\alpha =N^{-1}\) (more precisely, it is a consequence of~\cite[Equation~(3.13)]{Brydges+Frohlich+Spencer-1982}).

\begin{remark}
	The random walk representation of~\cite{Brydges+Frohlich+Spencer-1982} for the spin \(O(N)\) model gives~\ref{hyp:weak_SL} with \(\lambda=N^{-1}\) as parameter (at fixed \(\beta\)). Moreover, a similar argument as the one used in the proof of Lemma~\ref{lem:A4_for_XY} gives~\ref{hyp:J_path_lower_bnd} for the spin \(O(N)\) models with \(\lambda = N^{-1}\). We did not include the \(O(N)\) model in the discussion as the lack of correlation inequalities (most notably~\ref{hyp:sub_mult}) makes the whole discussion more complicated and less homogeneous.
\end{remark}

\subsubsection{Ising with a positive field}

\begin{lemma}
	\label{lem:weakSL_IPF}
	Let \(\beta\geq 0\), \(h>0\) and set \(\lambda = \beta\cosh(\beta) e^{-h}\).
	Then,
	\begin{equation}
		G^{\IsingPosFie}_{\beta,h}(x,y) \leq G^{\SAW}_{\lambda}(x,y) \leq G^{\KRW}_{\lambda}(x,y).
	\end{equation}
\end{lemma}
\begin{proof}
	The second inequality is trivial.
	The following proof is by no means self-contained.
	We refer to~\cite{Ott-2019} for notation and definitions of the objects.
	We use the same argument as in~\cite[Lemma 3.2]{Ott-2019} with the following replacement of the partitioning over clusters in the first current:
	the source constraint implies the existence of a self-avoiding path in \(\saw(x,y)\) using only edges with odd values of the current.
	The weight of such a path in finite volume is
	\begin{equation*}
		\frac{Z_{\Lambda_g\setminus \gamma}}{Z_{\Lambda_g}} \prod_{i=1}^{\abs{\gamma}} \sinh(\beta J_{\gamma_{i-1}\gamma_i}) \leq \prod_{i=1}^{\abs{\gamma}} \sinh(\beta J_{\gamma_{i-1}\gamma_i})
	\end{equation*}
	by the GKS inequality.
	A union bound, the same finite-energy argument as in~\cite[Lemma 3.2]{Ott-2019}, inclusion of sets and our normalization choice \(J_{ij}\leq 1\) yield the result.
\end{proof}

%
\subsection{Assumption~\ref{hyp:J_path_lower_bnd}}
%

The statement is immediate for \(\SAW\).

\subsubsection{GFF and KRW}

The desired inequality follows immediately from the identity~\eqref{eq:GFF_RW_Rep_Cov}, by restricting to self-avoiding trajectories of the random walk and imposing that \(T-1\) coincides with the time at which \(y\) is visited for the first time:
\[
	G_{\lambda}^{\GFF}(x,y)
	\geq
	\underbrace{\frac{m^2}{1+m^2}}_{\equiv c_\lambda} \sum_{\gamma\in\saw(x,y)}
	\Bigl( \underbrace {\frac{1}{1+m^2}}_{\equiv C_\lambda}\Bigr)^{\!\abs{\gamma}+1} \prod_{k=1}^{\abs{\gamma}} J_{\gamma_{k-1}\gamma_k}.
\]

\subsubsection{Potts model and FK percolation}

We prove~\ref{hyp:J_path_lower_bnd} for FK percolation and use the relation~\eqref{eq:Potts_FK_Corresp} between the Potts (and thus Ising) model and FK percolation to deduce the result for the Potts (and Ising) model.

Let \(\gamma=(u_0=x, u_1, \dots, u_{N-1}, u_N=y) \in \saw(x,y)\) and denote by \(\mathscr{O}_\gamma\) the event that the cluster of \(x\) (and \(y\)) is given by \(\gamma\):
\[
	\mathscr{O}_\gamma = \bigcap_{k=1}^N \{(u_{k-1},u_{k}) \text{ is open}\} \cap \bigcap_{k=0}^N \{(u_{k},v) \text{ is closed for all }v\neq u_{k-1},u_{k+1}\}.
\]
Since
\begin{equation}\label{eq:DeletionTolerance}
	\inf_{\eta\in\{0,1\}^{E_d}} \Phi^{\FK}_{\beta}(\omega_e=1 \given \omega_f=\eta_f \ \forall f\neq e) \geq \frac{e^{\beta J_e}-1}{e^{\beta J_e}-1+q}
	\geq \frac{\beta J_e}{e^\beta - 1 + q} ,
\end{equation}
where we used the fact that \(J_e\in [0, 1]\), it follows that
\begin{align*}
	\Phi^{\FK}_{\beta,q}(\mathscr{O}_\gamma)
	&\geq
	\prod_{k=1}^N \frac{e^{\beta J_{u_{k-1}u_k}}-1}{e^{\beta J_{u_{k-1}u_k}}-1+q}\cdot \Bigl( \prod_{v\neq 0} e^{-\beta J_{0v}} \Bigr)^{N+1} \\
	&\geq
	\underbrace{\frac{e^\beta-1+q}{\beta}}_{\equiv c_\lambda} \Bigl(\underbrace{\frac{\beta e^{-\beta}}{e^\beta-1+q}}_{\equiv C_{\lambda}}\Bigr)^{\!\!N+1} \cdot \prod_{k=1}^N J_{u_{k-1}u_k} ,
\end{align*}
thanks to the normalization assumption \(\sum_v J_{0v} = 1\).

\subsubsection{XY model}

\begin{lemma}
	\label{lem:A4_for_XY}
	The XY model satisfies \ref{hyp:J_path_lower_bnd} with \(c_{\lambda} = e^{-\beta
	}\) and \(C_{\lambda} = \frac{1}{2}\beta e^{-\beta}\).
\end{lemma}
\begin{proof}
	We work in finite volume \(\Lambda\) and take limits afterwards. Define \(S_x = (S_x^1,S_x^2) = (\cos(\theta_x),\sin(\theta_x))\), so \(\cos(\theta_i-\theta_j) = S_i\cdot S_j\). By symmetry,
	\begin{equation*}
		\mu_{\Lambda;\beta}^{\XY}\bigl(\cos(\theta_x-\theta_y)\bigr) = 2\mu_{\beta}^{\XY}\bigl(S_x^1S_y^1\bigr).
	\end{equation*}
	Moreover, denoting \(E_{\Lambda}=\{\{i,j\}\subset \Lambda\}\), Taylor expansion of the Boltzmann weight gives
	\begin{equation*}
		e^{\beta J_{ij} \cos(\theta_i-\theta_j) } = \sum_{n,m:E_{\Lambda}\to\bbZ_+} w(n)w(m) \prod_{i\in \Lambda}(S_i^1)^{I_i(n)} (S_i^2)^{I_i(m)},
	\end{equation*}
	where \(w(n) = \prod_{\{i,j\}\in E_{\Lambda}}\frac{(\beta J_{ij})^{n_{ij}} }{n_{ij}! }\) and \(I_i(n) = \sum_{j\neq i} n_{ij}\). Denote \(\partial n = \setof{i}{I_i(n) \text{ odd}}\).
	
	Use then
	\begin{equation*}
		\int_{0}^{2\pi} \cos(\theta)^a \sin(\theta)^b d\theta =
		\begin{cases}
		2\frac{\Gamma(\frac{a+1}{2})\Gamma(\frac{b+1}{2})}{\Gamma(\frac{a+b+2}{2})} = 2\sfB\bigl(\frac{a+1}{2}, \frac{b+1}{2}\bigr) & \text{ if } a \text{ and } b \text{ are even},\\
		0 & \text{ else},
		\end{cases}
	\end{equation*}
	where \(\Gamma\) denotes the Gamma function and \(\sfB\) the Beta function, to obtain
	\begin{multline*}
		Z_{\Lambda;\beta}^{\XY}\mu_{\Lambda;\beta}^{\XY}\bigl(S_x^1S_y^1\bigr) = 2^{|\Lambda|}\sum_{\partial n = \{x,y\},\,\partial m = \varnothing}w(n)w(m)\sfB\Bigl(\frac{I_x(n)+2}{2}, \frac{I_x(m)+1}{2}\Bigr)\times \\
		\times \sfB\Bigl(\frac{I_y(n)+2}{2}, \frac{I_y(m)+1}{2}\Bigr) \prod_{i\notin \{x,y\}} \sfB\Bigl(\frac{I_i(n)+1}{2}, \frac{I_i(m)+1}{2}\Bigr).
	\end{multline*}
	Define the weights
	\begin{equation*}
		W(n,m) = w(n)w(m) \prod_{i} \sfB\Bigl(\frac{I_i(n)+1}{2}, \frac{I_i(m)+1}{2}\Bigr).
	\end{equation*}
	
	Let now \(n\in (\bbZ_+)^{E_{\Lambda}}\) be such that \( \partial n = \{x,y\}\).
	A straightforward exploration argument plus loop erasure implies the existence of a self avoiding path \(\gamma\in \saw(x,y)\) such that \(n\) takes odd values on the edges of \(\gamma\).
	Fix some total order on \(\saw(x,y)\) and denote \(\eta=\eta(n)\) the smallest odd self avoiding path in \(n\), which we assimilate with the function taking value \(1\) on edges of \(\eta\) and \(0\) else.
	One can then uniquely decompose \(n\) as \((n-\eta(n))+\eta(n)\) with \(\partial(n-\eta(n)) = \varnothing\).
	For a function \(n\in (\bbZ_+)^{E_{\Lambda}}\) with \(\partial n = \varnothing\) and a path \(\gamma\in\saw(x,y)\), we write \(n\sim \gamma\) if \(\eta(n+\gamma)=\gamma\).
	Obviously, \(n\sim \gamma\) whenever \(n\) is zero on all edges sharing an endpoint with \(\gamma\).
	Let then \(\Phi_W\) be the probability measure on pairs \(n,m\in(\bbZ_+)^{E_{\Lambda}}\) with \(\partial n = \partial m =\varnothing\) defined by
	\begin{equation*}
		\Phi_W(n,m)\propto W(n,m)\IF{\partial n = \partial m =\varnothing}.
	\end{equation*}
	One finally obtains
	\begin{align*}
		\mu_{\beta}^{\XY} \bigl( S_x^1S_y^1 \bigr) &= \sum_{\gamma\in\saw(x,y)} \Phi_W\Bigl( \prod_{k=1}^{|\gamma|}\frac{\beta J_{\gamma_{k-1}\gamma_k}}{n_{\gamma_{k-1}\gamma_k}+1} \prod_{k=0}^{|\gamma|}\frac{\Gamma\bigl( \frac{I_{\gamma_k}(n)+1}{2}+1\bigr)\Gamma\bigl(\frac{I_{\gamma_k}(n)+I_{\gamma_k}(m)+2}{2}\bigr)}{\Gamma\bigl(\frac{I_{\gamma_k}(n)+1}{2}\bigr)\Gamma\bigl(\frac{I_{\gamma_k}(n)+I_{\gamma_k}(m)+2}{2}+1\bigr)} \IF{n\sim \gamma}\Bigr)\\
		&\geq C\sum_{\gamma\in\saw(x,y)} \prod_{k=1}^{|\gamma|}C\beta J_{\gamma_{k-1}\gamma_k} \Phi_W\Bigl( \prod_{k=0}^{|\gamma|}\IF{I_{\gamma_k}(n) = I_{\gamma_k}(m)= 0}\Bigr),
	\end{align*}
	where \(C=\frac{\Gamma( 3/2)\Gamma(1)}{\Gamma(1/2)\Gamma(2)} = \frac{1}{2}\). For a fixed \(\gamma\in\saw(x,y)\) and \(s\in[0,1]\), define \(Z_{\Lambda;\beta}^{\XY,s}\) the partition function of the XY model with coupling constants on edges touching sites of \(\gamma\) multiplied by \(s\). One then has
	\begin{multline*}
		\Phi_W\Bigl( \prod_{k=0}^{|\gamma|}\IF{I_{\gamma_k}(n) = I_{\gamma_k}(m)= 0}\Bigr) = \frac{Z_{\Lambda;\beta}^{\XY,0}}{Z_{\Lambda;\beta}^{\XY,1}} \\
		=
		\exp(-\int_{0}^1 ds \sum_{\{i,j\}\cap\gamma \neq \varnothing} \beta J_{ij}\mu_{\Lambda;\beta}^{\XY}(\cos(\theta_i-\theta_j)) )
		\geq
		e^{-\beta(|\gamma|+1)},
	\end{multline*}
	which concludes the proof.
\end{proof}

\subsubsection{Ising with a positive field}

\begin{lemma}
	\label{lem:LB_SAW_IPF}
	\(\IsingPosFie\) satisfies~\ref{hyp:J_path_lower_bnd} with \(c_{\lambda} = 1\) and \(C_{\lambda} = \frac{1}{2}\beta e^{-\beta}\lambda\).
\end{lemma}
\begin{proof}
	We again proceed in a non-self-contained manner and refer to~\cite{Ott-2019} for the notation.
	Let \(\Gamma\subset\saw(x,y)\).
	We start from the random-current representation of \(G_{\beta,h}^{\IsingPosFie}(x,y)\) in finite volume (see~\cite[(9)]{Ott-2019}).
	Restrict the sum over clusters of \(0\) that are SAWs in \(\Gamma\) to obtain
	\begin{equation*}
		G_{\beta,h}^{\IsingPosFie}(0,x)
		=
		\lim_{\Lambda\to\bbZ^d} \sum_{\gamma\in\Gamma} \prod_{i=1}^{\abs{\gamma}} \sinh(\beta J_{\gamma_i,\gamma_{i-1}}) \frac{Z_{\Lambda\setminus \bar{\gamma}}^2}{Z_{\Lambda}^2},
	\end{equation*}
	where \(\bar{\gamma}\) is the set of vertices in \(\gamma\) together with all edges touching \(\gamma\).
	
	For any fixed \(\gamma\), define then \(Z_{\Lambda,s,t}\) for \(s\in [0, 1],t\in [0, h]\) by multiplying the coupling constants of edges touching sites of \(\gamma\) by \(s\) and by setting the magnetic field at sites of \(\gamma\) to be \(t\).
	On then has
	\begin{equation*}
		\frac{Z_{\Lambda\setminus \bar{\gamma}}}{Z_{\Lambda}}
		=
		\frac{2^{-\abs{\gamma}}Z_{\Lambda,0,0}}{Z_{\Lambda,1,h}}
		=
		2^{-\abs{\gamma}}\frac{Z_{\Lambda,0,0}}{Z_{\Lambda,0,h}} \frac{Z_{\Lambda,0,h}}{Z_{\Lambda,1,h}}.
	\end{equation*}
	We can then differentiate/integrate to get
	\begin{equation*}
		\frac{Z_{\Lambda,0,0}}{Z_{\Lambda,0,h}}
		=
		\exp(-\int_0^h \sum_{i\in\gamma} \mu_{\Lambda,0,t}(\sigma_i) \dd t )\geq e^{-h\abs{\gamma}}.
	\end{equation*}
	In the same line of idea,
	\begin{equation*}
		\frac{Z_{\Lambda,0,h}}{Z_{\Lambda,1,h}}
		=
		\exp(-\int_{0}^1 \sum_{i\in\gamma, j\notin\gamma} \beta J_{ij} \mu_{\Lambda,s,h}(\sigma_i) \dd s )
		\geq
		e^{-\beta\abs{\gamma}}.
	\end{equation*}
	Combining all these, one gets
	\begin{equation*}
		G_{\beta,h}^{\IsingPosFie}(0,x)
		\geq
		\sum_{\gamma\in\Gamma} \prod_{i=1}^{\abs{\gamma}} 2^{-1}\beta J_{\gamma_i,\gamma_{i-1}} e^{-\beta}e^{-h}. \qedhere
	\end{equation*}
\end{proof}

\bibliographystyle{plain}
\bibliography{EZ_bib}

\end{document}